\documentclass{amsart}

\usepackage{amssymb}
\usepackage{amsmath}
\usepackage{amsthm}
\usepackage{amsbsy}
\usepackage{bm}
\usepackage{hyperref}
\newcommand{\comment}[1]{}

\theoremstyle{theorem}
    \newtheorem{theorem}{Theorem}
    \newtheorem{lemma}[theorem]{Lemma}

\theoremstyle{definition} 

    \newtheorem{remark}[theorem]{Remark}
    \newtheorem{example}[theorem]{Example}
    \newtheorem{exercise}[theorem]{Exercise}

\def\newblock{\hskip .11em plus .33em minus .07em}

\def\l{\left}
\def\r{\right}
\def\<{\langle}
\def\>{\rangle}

\newcommand{\one}{{\mathbf 1}}
\newcommand{\E}{\mbox{\bf E}}

\def\bar{\overline}
\def\P{{\bf P}}

\newcommand\tr{{\mbox{tr}}}

\newcommand\mnote[1]{} 
\newcommand\be{\begin{equation*}}

\newcommand\ee{\end{equation*}}

\newcommand\ben{\begin{equation}}
\newcommand\een{\end{equation}}
\newcommand\bes{\begin{eqnarray*}}
\newcommand\ees{\end{eqnarray*}}

\newcommand\bex{\begin{exercise}}
\newcommand\eex{\end{exercise}}
\newcommand\beg{\begin{example}}
\newcommand\eeg{\end{example}}
\newcommand\benu{\begin{enumerate}}
\newcommand\eenu{\end{enumerate}}
\newcommand\beit{\begin{itemize}}
\newcommand\eeit{\end{itemize}}
\newcommand\berk{\begin{remark}}
\newcommand\eerk{\end{remark}}
\newcommand\bdefn{\begin{defintion}}
\newcommand\edefn{\end{definition}}
\newcommand\bthm{\begin{theorem}}
\newcommand\ethm{\end{theorem}}
\newcommand\bprf{\begin{proof}}
\newcommand\eprf{\end{proof}}
\newcommand\blem{\begin{lemma}}
\newcommand\elem{\end{lemma}}

\newcommand{\sm}{{\raise0.3ex\hbox{$\scriptstyle \setminus$}}}

\def\l{\left}
\def\r{\right}

\def\CHI{\mathchoice%
{\raise2pt\hbox{$\chi$}}%
{\raise2pt\hbox{$\chi$}}%
{\raise1.3pt\hbox{$\scriptstyle\chi$}}%
{\raise0.8pt\hbox{$\scriptscriptstyle\chi$}}}
\def\smalloplus{\raise1pt\hbox{$\,\scriptstyle \oplus\;$}}

\begin{document}
\bibliographystyle{plain}

\title[]{Determinantal point processes in the plane from products of random matrices}

\author{Kartick Adhikari}

\address{Department of Mathematics,\\
        Indian Institute of Science,\\
        Bangalore 560012, India}

\email{kartickmath@math.iisc.ernet.in}

\author{Nanda Kishore Reddy}

\address{Department of Mathematics,\\
        Indian Institute of Science,\\
        Bangalore 560012, India}

\email{kishore11@math.iisc.ernet.in}

\author{Tulasi Ram Reddy}

\address{Department of Mathematics,\\
        Indian Institute of Science,\\
        Bangalore 560012, India}

\email{tulasi10@math.iisc.ernet.in}

\author{Koushik Saha}

\address{Department of Mathematics,\\
        University of California Davis\\
        Davis, CA 95616}

\email{ksaha@math.ucdavis.edu}

\thanks{Partially supported by UGC
(under SAP-DSA Phase IV). Research of Koushik Saha is partially supported by INSPIRE fellowship, Department of Science and Technology, Government of India. Research of Nanda Kishore reddy and Tulasi Ram Reddy is supported by CSIR-SPM fellowship, CSIR, Government of India. }

\date{\today}
\maketitle

\begin{abstract}
We show the density of eigenvalues for three classes of random
matrix ensembles is determinantal. First we derive the density of eigenvalues of product of $k$ independent $n\times n$ matrices with i.i.d. complex Gaussian entries with a few of matrices being inverted. In second example
we calculate the same for (compatible) product of rectangular matrices with i.i.d. Gaussian entries and in last example we calculate for product of independent truncated unitary random matrices. We derive exact expressions for limiting expected empirical spectral distributions of above mentioned ensembles.

\end{abstract}
\section{Introduction and Main results}
In this article we show that the eigenvalues of certain classes of random matrix ensembles form a determinantal point process on the complex plane. In particular, we have obtained the  density of the eigenvalues of these matrix ensembles.

The first  well known example  of determinantal point process from the field of random matrices  is \textit{circular unitary ensemble}, which  is the set of eigenvalues of a random unitary matrix sampled from the Haar measure on the set of all $n\times n$ unitary matrices, $\mathcal U(n)$.   Dyson \cite{dyson} introduced this ensemble and showed that the circular unitary ensemble forms a determiantal point process on $S^1$.

Ginibre \cite{ginibre} introduced three ensembles of matrices  with i.i.d. real, complex and quaternion  Gaussian entries respectively without imposing a Hermitian condition. He showed that the eigenvalues of an $n\times n$ matrix with  i.i.d. standard complex Gaussian entries  form a determinantal process on the complex plane.

\.{Z}yczkowski and  Sommers   \cite{sommer} generalised the result of Dyson \cite{dyson}. Let $U$ be a matrix drawn from the Haar distribution on  $\mathcal U(n)$. \.{Z}yczkowski and  Sommers    showed in \cite{sommer} that the eigenvalues of the left uppermost $m\times m$  block of  $U$ (where $m < n$) form a determinantal point process on $\mathbb D=\{z \in\mathbb C: |z|\leq 1\}$. They found the exact distribution of the eigenvalues and from there it follows that they form a determinantal point process.

Krishnapur \cite{manjunath} showed that the eigenvalues of $A^{-1}B$ form a determinantal point process on the complex plane when $A$ and $B$ are independent random matrices with i.i.d. standard complex Gaussian entries. In random matrix literature this matrix ensemble $A^{-1}B$ is known as \textit{spherical ensemble}.

 Akemann and Burda \cite{akemann} have derived the eigenvalue density for the product of $k$ independent $n\times n$ matrices with i.i.d. complex Gaussian entries.
 In this case the joint probability distribution of the eigenvalues of the product matrix is found to be given by a determinantal point process as in the case of Ginibre,
 but with a complicated weight given by a Meijer $G$-function depending on $k$.

Their derivation hinges on the generalized Schur decomposition for matrices and the method of orthogonal polynomials.
They  computed all eigenvalue density correlation functions exactly for finite $n$ and fixed $k$.
A similar kind of study has been done  on product of independent matrices with quaternion Gaussian entries  in \cite{ipsen}.

In a successive work, Akemann,  Keiburg and Wei \cite{akemann1}  showed that the singular values of  product of $k$ independent  random matrices with i.i.d. complex Gaussian entries form a determinantal point process on real line.
This generalises the classical Wishart-Laguerre Ensemble with $k=1$.
In a very recent work by Akemann, Ipsen and Keiburg \cite{akemann2},  a similar kind of result is proved for singular values of product of independent rectangular matrices with i.i.d. complex Gaussian entries. Here also the correlation functions are given by a determinantal point process, where the kernel can be expressed in terms of Meijer G-functions. For a detailed  discussion on determinantal point processes the reader can look at the survey articles \cite{manjusurvey},\cite{soshnikovsurvey}.

Now following the work of Krishnapur \cite{manjunath} on spherical ensembles and the work of Akemann and Burda \cite{akemann} on the product of $k$ independent  $n\times n$ Ginibre matrices, it is a natural question to ask, what can be said about the eigenvalues of  product of $k$ independent Ginibre matrices when a few of them are inverted? More precisely, do the eigenvalues of $A=A_1^{\epsilon_1}A_2^{\epsilon_2}
\cdots A_k^{\epsilon_k}$ form a determinantal point process, where each $\epsilon_i$
is $+1$ or $-1$ and $A_1,
A_2,\ldots, A_k$ are independent matrices with i.i.d. standard complex Gaussian entries? 
The answer is yes and the following theorem, our first result, answers it in detail.
\begin{theorem}\label{thm:result1}
Let $A_1,A_2,\ldots,A_k$ be independent $n\times n$ random matrices with i.i.d. standard complex Gaussian entries.
Then the eigenvalues of  $A=A_1^{\epsilon_1}A_2^{\epsilon_2}
\ldots A_k^{\epsilon_k}$, where each $\epsilon_i$
is  $+1$ or $-1$, form a determinantal point process. Equivalently, one can say that the vector of
eigenvalues of $A$ has density
(with respect to Lebesgue measure on $\mathbb{C}^{n}$) proportional to
$$
\prod_{\ell=1}^{n}\omega(z_{\ell})\prod_{i<j}^{n}|z_i-z_j|^2
$$
with a weight function $ \omega(z)$ , where
\begin{equation}\label{weight1}
|dz|^2\omega(z)=\int_{x_1^{\epsilon_1}\cdots x_{k}^{\epsilon_k}
=z}e^{\l(-\sum_{j=1}^{k}|x_{j}|^2\r)}
\prod_{j=1}^{k}|x_{j}|^{(1-\epsilon_j)(n-1)}\prod_{j=1}^k|dx_{j}|^2.
\end{equation}

\end{theorem}

We write
$$
|dz|^2\omega(z)=\int_{h(x_1,x_2,\ldots, x_k)=z}g(x_1,x_2,\ldots,x_k)|dx_1|^2|dx_2|^2\cdots |dx_k|^2
$$
if
$$
\int f(z)\omega(z)|dz|^2=\int f(h(x_1,x_2,\ldots ,x_k))g(x_1,x_2,\ldots,x_k)|dx_1|^2|dx_2|^2\cdots |dx_k|^2
$$
for all $f:\mathbb{C}\to \mathbb{C}$ integrable function.

In our next result we deal with the eigenvalues of product of $k$ independent rectangular matrices with i.i.d. complex Gaussian entries.
This is a generalization of result of Osborn \cite{osborn}, where he derives the eigenvalues density of product of two rectangular matrices.
Here also the eigenvalues form a determinantal point process. For a related work on the singular values on product of rectangular matrices see
\cite{akemann2}. The following theorem states our next result.

\begin{theorem}\label{thm:rectanguular}
Let $A_1,\; A_2,\ldots,\;A_k$ be independent rectangular matrices of dimension $n_i\times n_{i+1}$ for $i= 1,2,\ldots,k$,
with $n_{k+1}=n_1=\min\{n_1,n_2,\ldots,n_k\}$ and with i.i.d. standard complex Gaussian entries. Then the eigenvalues $z_1,z_2,\ldots,z_{n_1}$
 of
$A=A_1A_2\cdots A_k$ form a determinantal
point process on the complex plane.  Equivalently, one can say that the vector of eigenvalues of
$A=A_1A_2\cdots A_k$ has density
(with respect to Lebesgue measure on $\mathbb{C}^{n_1}$) proportional to
$$
\prod_{\ell=1}^{n_1}\omega(z_{\ell})\prod_{i<j}^{n_1}|z_i-z_j|^2
$$
with a weight function
$$
|dz|^2\omega(z)=\int_{x_{1}\cdots x_{k}=z}e^{-\sum_{j=1}^{k}|x_{j}|^2}
\prod_{j=1}^{k}|x_{j}|^{2(n_j-n_1)}\prod_{j=1}^k|dx_{j}|^2.
$$
\end{theorem}

 Our last example deals with product of truncated unitary matrices. This is a generalisation of a result on truncated unitary matrix by  \.{Z}yczkowski and Sommers  \cite{sommer}. We show that the eigenvalues of product of truncated unitary matrices and inverses are also determinantal. The following theorem states it precisely.

 \begin{theorem}\label{thm:result3}
Let $U_1,U_2,\ldots,U_k$ be $k$ independent  Haar distributed unitary matrices of dimension $n_i\times n_i$ ($n_i>m$) for $i=1,2,\ldots,k$ respectively and $A_1, A_2,\ldots,A_k$ be $m\times m$ left uppermost blocks of $U_1, U_2,\ldots,U_k$ respectively.
 Then the eigenvalues $z_1, z_2,\ldots,z_m$ of $A=A_1^{\epsilon_1}A_2^{\epsilon_2}\cdots A_k^{\epsilon_k}$,
where each $\epsilon_i$ is $+1$ or $-1$ form a determinantal point process on the complex plane.
 Equivalently one can say that the vector of eigenvalues of $A=A_1^{\epsilon_1}A_2^{\epsilon_2}
 \cdots A_k^{\epsilon_k}$ has density
(with respect to Lebesgue measure on $\mathbb{C}^{m}$) proportional to
$$
\prod_{\ell=1}^{m}\omega(z_{\ell})\prod_{1\le i<j\le m}|z_i-z_j|^2
$$
with a weight function
$$
|dz|^2\omega(z)=\int_{x_1^{\epsilon_1}\cdots x_{k}^{\epsilon_k}
=z}\prod_{j=1}^{k}(1-|x_{j}|^2)^{n_j-m-1}|x_j|^{(m-1)(1-\epsilon_j)}\one_{\{|x_{j}|\le 1\}}(x_{j})|dx_{j}|^2.
$$
\end{theorem}

We organize this paper as follows. In Section \ref{product:ginibre}, we prove Theorem \ref{thm:result1} leaving some technical details for Appendix section. In Section \ref{product:rectangular}, we prove Theorem \ref{thm:rectanguular} and in Section \ref{product:truncatedunitary}, we prove Theorem \ref{thm:result3} for the case
where all  $n_i$ are equal, because the general case is just few notational changes away from there.  In Section \ref{kernel}, we calculate the kernels of the determinantal point processes which have emerged  in Theorem \ref{thm:result1}, Theorem \ref{thm:rectanguular}  and Theorem \ref{thm:result3}.

In Section \ref{limiting spectral distribution}, using these kernels, we identify the limit of expected  empirical distribution of these matrix ensembles. In particular, in Theorem \ref{lsd:ginibre}, we calculate the limiting expected empirical distribution of square radial part of eigenvalues of product of Ginibre and inverse Ginibre matrices. Since one point corelation function of the corresponding point process, which gives the expected empirical spectral distribution, dose not depend on the angular part of the eigenvalues, the limiting distribution of the radial part identify the limiting spectral distribution completely. For limit of expected  spectral distribution of product of independent matrices with independent entries, one can see \cite{sean_soshnikov}, \cite{gotze_tikhomirov}.

 In Theorem \ref{lsd:rectangular} of Section \ref{limiting spectral distribution}, we calculate the limit of expected empirical distribution of square radial part of eigenvalues of product of rectangular matrices with independent complex Gaussian entries. Limit of empirical spectral distribution  of product of independent rectangular matrices has been derived in \cite{burda_nowak}, but the limiting density is obtained in terms of M-transform. However, we have a simple explicit expression in terms of uniform distribution for the limit. Finally in Theorem \ref{lsd:unitary} we calculate the same for product of truncated unitary matrices.

In Section \ref{appendix}, we give details of some technicalities  of the proof of Theorem \ref{thm:result1} and Theorem \ref{thm:rectanguular}, and a little discussion on theory of manifolds.

 All our proofs in this article rely greatly on generalized Schur decomposition of product of  matrices. We describe this decomposition at the end of this section. We also use   RQ decomposition and QR decomposition in the proof of Theorem \ref{thm:rectanguular} and Theorem \ref{thm:result3} respectively.
 We discuss these decompositions briefly in Section \ref{appendix}. More details on these decompositions
 can be found in \cite{forrester}, \cite{muirhead} and \cite{srivastava}.\\

 \noindent{\bf  Generalized Schur-decomposition:} Any $n\times n$ square matrices $A_1,A_2, \ldots,A_k$ can be written as
 \begin{eqnarray*}
 A_1&=&U_1S_1U_2^*,\\
 A_2&=&U_2S_2U_3^*,\\
 &\vdots &\\
 A_{k-1}&=&U_{k-1}S_{k-1}U_k^*,\\
 A_{k}&=&U_{k}S_{k}U_1^*,
 \end{eqnarray*}
 where $U_i\in \mathcal{U}(n)/(\mathcal{U}(1))^n$ and $S_1,S_2,\ldots,S_k$ are upper triangular matrices. For details, see Appendix.

\section{Product of Ginibre matrices and inverse Ginibre matrices}\label{product:ginibre}

We begin this section with some remarks on Theorem \ref{thm:result1}.
\begin{remark}
(i)  If $k=2$, $\epsilon_1=-1$ and $\epsilon_2=1$, then  from \eqref{weight1} we get that
$$|dz|^2\omega(z)=\int_{\frac{x_{2}} {x_{1}}=z}e^{- (|x_{1}|^{2}+|x_{2}|^2)}
|x_{1 }|^{2(n-1)}|dx_{1}|^2|dx_{2}|^2=c \frac{|dz|^2}{(1+|z|^2)^{(n+1)}},
$$
with some constant $c$. Hence the density of the eigenvalues of $A_1^{-1}A_2$
is proportional to
$$
\prod_{i=1}^{n}\frac{1}{(1+|z_i|^2)^{n+1}}\prod_{i<j}|z_{i}-z_j|^2.
$$
From the above expression it is clear that the eigenvalues of $A_1^{-1}A_2$ form a determinantal point process in a complex plane. This result was  proved by Krishnapur in \cite{manjunath} using a different technique.\\

\noindent(ii)  If $\epsilon_i=1$ for $i=1,2,\ldots,k$, then by Theorem \ref{thm:result1} it follows that the eigenvalues of
$A_1A_2\ldots A_k$  form a determinantal point process. This result is due to Akemann and Burda \cite{akemann}.\\

\noindent(iii) If $\epsilon_i=-1$ for $1\le i\le p$ and  $\epsilon_i=1$
for $p+1\le i\le k$, then weight function is given by
$$
|dz|^2\omega(z)=\int_{\frac{x_{p+1}\cdots x_k}{x_1x_2\cdots x_p}
=z}e^{\l(-\sum_{j=1}^{k}|x_{j}|^2\r)}
\prod_{j=1}^{p}|x_{j}|^{2(n-1)}\prod_{j=1}^k|dx_{j}|^2.
$$
\end{remark}
Now we proceed to prove Theorem \ref{thm:result1}. We shall denote
\begin{equation}
(Dx)=(dx_1,dx_2,\ldots, dx_n) \mbox{ and } |Dx|=\bigwedge_{i=1}^n(dx_i \wedge \bar{dx_i})
\end{equation}
for a complex vector $x=(x_1,x_2,\ldots,x_n)$ and
\begin{equation}\label{def:DA}|DA|=\bigwedge_{i,j}(dA(i,j)\wedge d\bar{A(i,j)})\end{equation}
 for a complex matrix $A$. Here wedge product is taken only for the non-zero variables of matrix $A$.

\begin{proof}[Proof of Theorem \ref{thm:result1}]
The density of $(A_1,A_2,\ldots,A_k)$ is proportional to
$$
\prod_{\ell=1}^{k}e^{-\tr( A_{\ell}A_{\ell}^*)}\bigwedge_{\ell=1}^{k}\bigwedge_{i,j=1}^n|dA_{\ell}(i,j)|^2
$$
where $|dA_{\ell}(i,j)|^2=dA_{\ell}(i,j)\wedge d\bar{A}_{\ell}({i,j}).$ Actually, here the proportional constant is
$\frac{1}{\pi^{kn^2}}$, but to make life less painful for ourselves, we shall omit constants in every step to follow.
Since we are dealing with probability measures, the  constants
can be recovered at the end
by finding normalization constants.

Now by generalized Schur-decomposition (see \eqref{eqn:square}), we have
$$
A_i^{\epsilon_i}=U_iS_i^{\epsilon_i}U_{i+1}^*,\;\;i=1,2,\ldots,k,\; \mbox{and}\;\; k+1=1,
$$
where $S_1,S_2,\ldots,S_k$ are upper triangular matrices
and  $U_1,U_2,\ldots,U_k$ are unitary matrices with $U_{k+1}=U_1$. Let the diagonal entries of $S_i$ be $(x_{i1},x_{i2},\ldots,x_{in})$.
One can see that eigenvalues $z_1,z_2,\ldots,z_n$ of $A=A_1^{\epsilon_1}A_2^{\epsilon_2}\ldots A_k^{\epsilon_k}$ are given by
$$
z_j=\prod_{i=1}^{k}x_{ij}^{\epsilon_i},\;\;j=1,2,\ldots,n.
$$
Now, by using Jacobian determinant calculation for generalised Schur decomposition (see \eqref{eqn:square}),
we get
\begin{equation}\label{ginibre:jacobian:relation}
\prod_{i=1}^{k}|DA_i^{\epsilon_i}|=|\Delta({\underline{z}})|^2\prod_{i=1}^{k}|DS_i^{\epsilon_i}|\prod_{i=1}^{k}|dH(U_i)|,
\end{equation}
where $\Delta({\underline{z}})=\prod_{i<j}(z_i-z_j)$. From \eqref{eqn:wedge:relation}, it is easy to see that
\begin{equation}\label{DxA}
|(Dx)A|=|A(Dx)^t|=|\det(A)|^2|Dx|\end{equation}
for any complex matrix $A$ and vector $x$.
Since
$$  |DA^{-1}|=|A^{-1}(DA)A^{-1}|,$$
using \eqref{DxA}, we have
\begin{equation}\label{DA:inverse}
|DA_i^{\epsilon_i}|=|\det(A_i)|^{4n(\frac{\epsilon_i-1}{2})}|DA_i|.\end{equation}
By similar calculation for upper triangular matrices $S_i$, we get
\begin{equation}\label{DS:inverse}
|DS_i^{\epsilon_i}|=|\det(S_i)|^{2(n+1)(\frac{\epsilon_i-1}{2})}|DS_i|.
\end{equation}
Now using \eqref{ginibre:jacobian:relation}, \eqref{DA:inverse} and \eqref{DS:inverse}, and since $
|\det(S_i)|=|\det(A_i)|$,
we get
\begin{eqnarray}\label{eqn:van1}
\prod_{i=1}^{k}|DA_i|=|\Delta({\underline{z}})|^2\prod_{i=1}^k|\det(S_i)|^{(1-\epsilon_i)(n-1)}\prod_{i=1}^{k}|DS_i||dH(U_i)|.
\end{eqnarray}
The density  of $A_1,A_2,\ldots,A_k$ can be written
in new variables as
$$
|\Delta(\underline{z})|^2\prod_{i=1}^ke^{-tr(S_iS_i^*)}|\det(S_i)|^{(1-\epsilon_i)(n-1)}|DS_i||dH(U_i)|.
$$

By integrating out  the  non-diagonal entries of $S_1,S_2,\ldots,S_k$, we get the density of
diagonal entries of  $S_1,S_2,\ldots,S_k$ to be proportional to
$$
|\Delta(\underline{z})|^2\prod_{i=1}^k\prod_{j=1}^{n}e^{-|S_i(j,j)|^2}|S_i(j,j)|^{(1-\epsilon_i)(n-1)}|dS_i(j,j)|^2.
$$
Hence the  density of $z_1,z_2,\ldots,z_{n}$ is proportional to
$$
\prod_{\ell=1}^{n}\omega(z_{\ell})\prod_{i<j}^{n}|z_i-z_j|^2
$$
with a weight function
$$
|dz|^2\omega(z)=\int_{x_1^{\epsilon_1}\cdots x_{k}^{\epsilon_k}
=z}e^{\l(-\sum_{j=1}^{k}|x_{j}|^2\r)}
\prod_{j=1}^{k}|x_{j}|^{(1-\epsilon_j)(n-1)}\prod_{j=1}^k|dx_{j}|^2.
$$
This completes the proof of the theorem.
\end{proof}

\section{product of rectangular matrices}\label{product:rectangular}
In this section we will prove Theorem \ref{thm:rectanguular} borrowing some results from Section \ref{appendix}. Before that we make a remark on the assumption of Theorem \ref{thm:rectanguular}.
\begin{remark}
The condition $n_1=\min\{n_1,n_2,\ldots,n_k\}$ in Theorem \ref{thm:rectanguular} is taken for simplicity. Since we want to calculate density of
non-zero eigenvalues  of product of (compatible) rectangular matrices  $A_1A_2\ldots A_k$ and
the set of non-zero eigenvalues of
$A_1A_2\ldots A_k$ remains unaltered for any rotational combination of
$A_1,A_2,\ldots,A_k$. 
So the set of non-zero eigenvalues of $A_1A_2\ldots A_k$ is less or equal to $\min\{n_1,n_2,\ldots,n_k\}$.  Therefore, we can assume that $n_1$ is the minimum among $n_1,n_2,\ldots,n_k$.
\end{remark}

\begin{proof}[Proof of Theorem \ref{thm:rectanguular}]
Density of $A_1,A_2,\ldots,A_k$ is proportional to
$$
\prod_{i=1}^{k}e^{-tr(A_iA_i^*)}|DA_i|
$$
where $|DA_i| $ is as defined in \eqref{def:DA}.
Now using the transformations as discussed in Remark \ref{re:tran} and also using (\ref{eqn:jacobian}), the density of $A_1,A_2,\ldots,A_k$ can be written
in new variables as
$$
|\Delta(Z)|^2\prod_{i=1}^ke^{-tr(S_iS_i^*+B_iB_i^*)}|\det(S_i)|^{2(n_{i+1}-n_1)}|DS_i||dH(U_i)||DB_i|
$$
where $B_i, S_i$ and $dH(U_i)$ are as in \eqref{transformation_i+2} and \eqref{dsf:DS_i} respectively   and
$$
\Delta(Z)=\prod_{i<j}^{n_1}(z_i-z_j)\; \mbox{and}\;z_j=\prod_{i=1}^{k}S_i(j,j)\;\mbox{for}\;j=1,2,\ldots,n_1.
$$
We take $B_1=0$ and $|DB_1|=1$.
By integrating out the variables in $B_2,\ldots,B_k$, $U_1,U_2,\ldots,U_k$ and the  non-diagonal entries of $S_1,S_2,\ldots,S_k$, we get the density of
diagonal entries of  $S_1,S_2,\ldots,S_k$ to be proportional to
$$
|\Delta(Z)|^2\prod_{i=1}^k\prod_{j=1}^{n_1}e^{-|S_i(j,j)|^2}|S_i(j,j)|^{2(n_{i+1}-n_1)}|dS_i(j,j)|^2.
$$
Hence the  density of $z_1,z_2,\ldots,z_{n_1}$ is proportional to
$$
\prod_{\ell=1}^{n_1}\omega(z_{\ell})\prod_{i<j}^{n_1}|z_i-z_j|^2
$$
with a weight function
$$
|dz|^2\omega(z)=\int_{z_{1}\cdots z_{k}=z}e^{\l(-\sum_{j=1}^{k}|z_{j}|^2\r)}
\prod_{j=1}^{k}|z_{j}|^{2(n_j-n_1)}\prod_{j=1}^k|dz_{j}|^2.
$$
This completes the proof of the theorem.
\end{proof}

\section{Product of truncated unitary matrices}\label{product:truncatedunitary}
Before we prove Theorem \ref{thm:result3}, we take a look at a special case of it, which covers the result of \.{Z}yczkowski and   Sommers  \cite{sommer}
on a single truncated unitary matrix.

\begin{remark}\label{thm:truncatedunitary}
(i). If $k=1$ and $\epsilon_1=1$, then  Theorem \ref{thm:result3} says that
the eigenvalues $z_1,z_2,\ldots,z_m$ of $A_{m\times m}$ left-upper block of Haar distributed unitary matrix
$$
 U_{n\times n}=
            \left[ {\begin{array}{cc}
             A_{m\times m} & B_{m\times n-m} \\
             C_{n-m\times m} & D_{n-m\times n-m} \\
                \end{array} } \right]
 $$
 form a determinantal point process with density proportional to
 \begin{equation}\label{unitary:specialcase}
 \prod_{1\le j<k \le m}|z_j-z_k|^2\prod_{i=1}^{m}(1-|z_i|^2)^{n-m-1}\one_{\{|z_i|\le 1\}}(z_i).
 \end{equation}
 This special case was proved by \.{Z}yczkowski and  Sommers  in \cite{sommer}.
 Here we present another proof of this special case which is  slightly different from previous two. We also hope that the proof of this special case \eqref{unitary:specialcase} will help the reader to understand the proof of Theorem \ref{thm:result3} in a better way.
 \vspace{.5cm}\\
(ii). For simplicity we have taken $m\times m$  left-upper blocks of matrices. But we can take  any $m\times m$ blocks of matrices, because their probability distributions are same.

\end{remark}

In proving \eqref{unitary:specialcase} and Theorem \ref{thm:result3}, we need to introduce some basic notation and facts. Let $\mathcal{M}_n$ be the space of all
 $n\times n$ complex matrices equipped with Euclidean norm, $\|M\|=\sqrt{\tr(M^*M)}$.
Let  $\mathcal{U}(n)$ be the space of all $n\times n$ unitary matrices. It is a manifold of
 dimension $n^2$ in $\mathbb{R}^{2n^2}$.
 Haar measure on $\mathcal{U}(n)$ is normalized volume measure on manifold $\mathcal{U}(n)$
 which is denoted by $H_{\mathcal{U}(n)}$. For a detailed discussion on this, see \cite{andersonbook}.
 Define
\begin{equation}\label{n_m,n}
\mathcal{N}_{m,n}=\{Y\in \mathcal{M}_n:Y_{i,j}=0,1\le j<i\le m \},\end{equation}
\begin{equation}\label{v_m,n}
\mathcal{V}_{m,n}=\mathcal{N}_{m,n}\cap \mathcal{U}(n).
\end{equation}
Let
$H_{\mathcal{V}}$ be the normalized volume measure on manifold $\mathcal{V}_{n,m}$.
But we suppress subscripts $m,n$ when there is no confusion.\\

\noindent\textbf{Weyl chamber:} This is a subset of $\mathbb C^n$ and is defined as $$\mathcal{W}_n:=\{(z_1,z_2,\ldots, z_n):z_1\ge z_2\ge \cdots\ge z_n\}\subset \mathbb{C}^n$$
where $z\ge w$ if $\Re(z)> \Re (w)$ or $\Re (z)=\Re(w)$ and $\Im(z)\ge \Im(w)$. The metric on Weyl chamber is given by
 $$
 \|\underline{z}-\underline{w}\|_{\mathcal{W}}=\min_{\sigma}\sqrt{\sum_{i=1}^n|z_i-w_{\sigma{i}}|^2},
 $$
 minimum is taken over all permutations of $\{1,2,\ldots ,n\}$. Weyl chamber with this metric is a polish space. We take the space of eigenvalues of $n\times n$ matrices as Weyl chamber through the following map
$\Phi_n:(\mathcal{M}_n,\|\cdot\|) \to  (\mathcal{W}_n,\|\cdot\|_{\mathcal{W}})$ which is defined as
\begin{equation}\label{Phi_n}\Phi_n(M)=(z_1,z_2,\ldots,z_n)\end{equation}
where $z_1\ge z_2\ge \cdots\ge z_n$ are the eigenvalues of $M$. The map $\Phi_n$ is a continuous map.
This can be seen from the fact that roots of complex polynomial are continuous functions
of its coefficients and eigenvalues are roots of characteristic polynomial whose coefficients
 are continuous functions of matrix entries.

Note that the map
\begin{equation}\label{Psi_n}\Psi_{n,m}  :(\mathcal{M}_n,\|\cdot\|) \to  (\mathcal{M}_m,\|\cdot\|), \  (n\ge m),\end{equation}
  taking every matrix to its $m\times m$ left uppermost block is also continuous.\\

\noindent {\bf A brief outline  of the proof of \eqref{unitary:specialcase}:}
The vector of eigenvalues $\underline{Z}=(z_1,z_2,\ldots,z_m)$ of $m\times m$ left uppermost block of a Haar distributed $n\times n$
unitary matrix define a measure
$\mu$ on $\mathcal{W}_m$. In other words $\underline{Z}$ is $\mathcal{W}_m$-valued random variable distributed according to $\mu$.  We show that there exists a function $p(z_1,z_2\ldots,z_m)$ such that expectation
of any complex valued bounded continuous function $f$ on $(\mathcal{W}_m,\|\cdot\|_{\mathcal{W}})$ is given by
\begin{eqnarray*}
\E[f(\underline{Z})]=\int_{\mathcal{W}_m}f(\underline{z})p(\underline{z})
|d\underline{z}|^2=\int_{\mathbb{C}^m}\frac{1}{m!}f(\underline{z})p(\underline{z})|d\underline{z}|^2
\end{eqnarray*}
where $p$ and $f$ are extended to $\mathbb{C}^m$ by defining
$$f(z_1,z_2,\ldots,z_m)=f(z_{(1)},z_{(2)},\ldots,z_{(m)}), \
p(z_1,z_2,\ldots,z_m)=p(z_{(1)},z_{(2)},\ldots,z_{(m)}),$$
where $\{z_1,z_2,\ldots,z_m\}=\{z_{(1)},z_{(2)},\ldots,z_{(m)}\}$  and $z_{(1)}\ge z_{(2)}\ge \cdots \ge z_{(n)}$ and
 $|d\underline{z}|^2$
is Lebesgue measure on $\mathbb{C}^m$.

So $p(\underline{z})$ gives the joint probability density of eigenvalues
$\underline{Z}$.  Also note that the set of symmetric continuous functions on $\mathbb{C}^m$ are in natural bijection
with set of continuous functions on $(\mathcal{W}_m,\|\cdot\|_{\mathcal{W}})$.

To compute   the above expectation, we approximate Haar measure on  $\mathcal{U}(n)$ by normalised Lebesgue measure
on its open neighbourhood in $\mathcal{M}_n $ (see Lemma \ref{lem:haarmeasure}).
We apply Schur decomposition to $m\times m$ left uppermost block and
integrate out the unitary matrix variables that come from Schur decomposition. By de-approximating (shrinking the neighbourhood), we get back to
$\mathcal{V}_{m,n}$, a sub-manifold of $\mathcal{U}(n)$ (using Lemma \ref{lem:manifoldmeasure} ).
 Then we integrate out auxiliary variables using Co-area formula to
 arrive at joint probability density of $\underline{Z}$.

Now we state Lemma \ref{lem:haarmeasure} and Lemma \ref{lem:manifoldmeasure}. We prove them at the end of this section.
 \begin{lemma}\label{lem:haarmeasure}
 Let $f:\mathcal{M}_n\to \mathbb{C}$ be a continuous function. Then
 $$
 \int f(U) dH_{\mathcal{U}(n)}(U) = \lim_{\epsilon \to 0} \frac{\int_{\|X^*X-I\|<\epsilon}
 f(X)dX }{\int_{\|X^*X-I\|<\epsilon} dX },
 $$
 where $dX$ and $\|\cdot\|$ denote differential element of volume measure  and   Euclidean norm  on $\mathcal M_n$, manifold
 of $n\times n$ complex matrices respectively and $H_{\mathcal{U}(n)}$ is the normalized volume measure on manifold $\mathcal{U}(n)$.
 \end{lemma}

 \begin{lemma}\label{lem:manifoldmeasure}
 Let $f:\mathcal{M}_n\to \mathbb{C}$ be a continuous function. Then
 $$
 \int f(V) dH_{\mathcal{V}}(V) = \lim_{\epsilon \to 0} \frac{\int_{\|X^*X-I\|<\epsilon}
  f(X)dX }{\int_{\|X^*X-I\|<\epsilon} dX },
 $$
 where $dX$ and $\|\cdot\|$ denote differential element of volume measure  and   Euclidean norm  on $\mathcal{N}_{m,n}$
 respectively and $H_{\mathcal{V}}$ is the normalized volume measure on manifold $\mathcal{V}_{n,m}$.
 \end{lemma}
\begin{proof}[Proof of \eqref{unitary:specialcase}]
For the sake of simplicity we shall use the same symbol $f$ for all $f,\ f\circ \Phi_m,\ f\circ \Phi_m\circ\Psi_{n,m}$ (when no confusion can arise) where $\Phi_m, \Psi_{n,m}$ are as defined in \eqref{Phi_n} and \eqref{Psi_n}. Now
\begin{eqnarray*}
\E [f(\underline{Z})] &=&\E[f\circ \Phi_m(A)]\\
&=&\E[f\circ \Phi_m\circ \Psi_{n,m}(U)]\\
&=&\int f(U)dH_{\mathcal{U}(n)}(U)\\
&=&\lim_{\epsilon\to 0}\frac{\int_{\|X^*X-I\|<\epsilon}f(X)dX}{\int_{\|X^*X-I\|<\epsilon}dX}
\;\;\mbox{(by Lemma \ref{lem:haarmeasure}).}
\end{eqnarray*}
We apply Schur-decomposition to $m\times m$ left upper most block of $X$ to get
\begin{eqnarray*}
X=\l[\begin{array}{lr}Q&0\\0&I\end{array}\r]Y\l[\begin{array}{lr}Q^*&0\\0&I\end{array}\r]
\end{eqnarray*}
where $Q$ is $m\times m $ unitary matrix, $Y_{i,j}=0$ for $1\le j<i\le m$. Usually one has to  choose some ordering on first $m$ diagonal entries of $Y$.
But since we are dealing with  the ratio of two integrals involving the same decomposition of $X$, so any ordering
chosen on diagonal entries of $Y$ will give us  same final result. So there is no need of ordering the diagonal entries.

Variables in $Q$  can be integrated out in both numerator and denominator, and they will cancel  each other.
So integrating out $Q$ variables we get that
\begin{eqnarray*}
\E[f(\underline{Z})]&=&\lim_{\epsilon\to 0}\frac{\int_{\{\|Y^*Y-I\|<\epsilon\}}
\prod_{i<j\le m}{|Y_{i,i}-Y_{j,j}|^2}f(Y)dY}{\int_{\{\|Y^*Y-I\|<\epsilon\}}
\prod_{i<j\le m}{|Y_{i,i}-Y_{j,j}|^2}dY}
\\&&[ Y\in\mathcal{N}_{m,n},\;dY \mbox{denotes Lebesgue measure on $\mathcal{N}_{m,n}$ }]
\\
&=&C\int |\Delta(\underline{z})|^2f(V)dH_{\mathcal{V}}(V)\;\;\;\mbox{[by Lemma \ref{lem:manifoldmeasure}]}
\end{eqnarray*}
where $\mathcal{V}=\mathcal V_{m,n}=\mathcal{N}_{m,n}\cap \mathcal{U}(n)$,
$C^{-1}=\int |\Delta(\underline{z})|^2dH_{\mathcal{V}}(V)$,  $z_i=V_{i,i}$ for $ i=1,2,\ldots,m,$   $V\in \mathcal{V}$,
 $\Delta(\underline{z})= \prod_{1\le i<j \le m}(z_i-z_j)$ is vandermonde determinant term. We shall use symbol $C$ for all numerical constants. Let
 $$\mathcal{V}_0= \{V_{n\times m}:V^*V=I, V_{ij}=0 \; \forall\; 1\le j<i\le m\}$$
and  $g:\mathcal{V}\to \mathcal{V}_0$ be projection map such that $g(V)$ is a matrix of dimension $n\times m$  by removing last $n-m$ columns from $V$. Now by Co-area formula \eqref{eqn:coarea},
\begin{equation}\label{coarea:1st}
\int |\Delta(\underline{z})|^2f(\underline{z}) dH_{\mathcal{V}}(V)=\int \left( \int |\Delta(\underline{z})|^2f(\underline{z})
dH_{g^{-1}(V_0)}(V)\right)dH_{\mathcal{V}_0}(V_0).
\end{equation}
For a fixed $V_0\in \mathcal{V}_0$ (so $z_1,z_2,\ldots,z_{m}$ are also fixed), ${g^{-1}(V_0)}$
 is a sub-manifold of  $\mathcal{V}$. It is isometric to the set of unit vectors in $\mathbb{C}^n$ which are orthogonal to  $m$
  columns of $V_0$. So $g^{-1}(V_0)$ is isometric to the manifold $\mathcal{U}(n-m)$. Jacobian in the Co-area formula for projection maps is equal to one.
So from \eqref{coarea:1st}, we get
\begin{eqnarray*}
\E [f(\underline{Z})]=C\int |\Delta(\underline{z})|^2f(\underline{z}) dH_{\mathcal{V}_0}(V_0)
\end{eqnarray*}
where
$z_{i}=V_0(i,i)$.
 Note that  $\mathcal{V}_0$ is a manifold of dimension $2nm-2m^2+m$ in $\mathbb{R}^{2nm-m^2+m}$ and its normalized volume
 measure is denoted by $H_{\mathcal{V}_0}$. Similarly we define
 $$\mathcal{V}_i= \{V_{n\times m-i}: V^*V=I,
 V_{s,t}=0 \; \forall\; 1\le s<t\le m\}$$
  and denote its normalized volume measure  by $H_{\mathcal{V}_i}$.
 Here also we denote  $V_i(\ell,\ell)$ by  $z_{\ell}$, where $V_i \in \mathcal{V}_i$.

Let $g_0:\mathcal{V}_0\to \mathcal{V}_1$ be projection map such that $g_0(V_0)$ is a matrix of dimension $n\times (m-1)$  by removing last column from $V_0$. Again by Co-area formula
\begin{eqnarray*}
\int |\Delta(\underline{z})|^2f(\underline{z}) dH_{\mathcal{V}_0}(V_0)=\int \left( \int |\Delta(\underline{z})|^2f(\underline{z})
dH_{g_0^{-1}(V_1)}(V)\right)dH_{\mathcal{V}_1}(V_1).
\end{eqnarray*}
For a fixed $V_1\in \mathcal{V}_1$ (so $z_1,z_2,\ldots,z_{m-1}$ are also fixed), ${g_0^{-1}(V_1)}$
 is a sub-manifold of  $\mathcal{V}_0$. It is isometric to the set of unit vectors in $\mathbb{C}^n$ which are orthogonal to $m-1$
  columns of $V_1$ whose $m$-th coordinates are zero. So $g_0^{-1}(V_1)$ is isometric to the manifold
  $$\mathcal{T}_1=\{(z_m,a_1,a_2,\ldots,a_{n-m})\in\mathbb{C}^{n-m+1}: |z_m|^2+\sum_{i=1}^{n-m} |a_i|^2 =1\}.$$
 When integrating $|\Delta(\underline{z})|^2f(\underline{z})$
 with respect to $H_{g_0^{-1}(V_1)}$, because of $z_m$ being the only $\underline{Z}$ variable involved, we get
\begin{eqnarray*}
\int |\Delta(\underline{z})|^2f(\underline{z}) dH_{\mathcal{V}_0}(V_0)=\int |\Delta(\underline{z})|^2f(\underline{z})
 dH_{\mathcal{T}_1}dH_{\mathcal{V}_1}(V_1).
\end{eqnarray*}
Now, by integrating out $a_1,a_2,\ldots,a_{n-m}$, we get
\begin{equation}\label{eqn:coarea1}
\int |\Delta(\underline{z})|^2f(\underline{z}) dH_{\mathcal{V}_0}(V_0)=C \int |\Delta(\underline{z})|^2
f(\underline{z})(1-|z_m|^2)^{n-m-1}\one_{\{|z_m|\le 1\}}(z_m) dH_{\mathcal{V}_1}(V_1)|dz_m|^2.
\end{equation}
Again by applying Co-area formula on the right hand side of \eqref{eqn:coarea1} and using similar argument as above, we get that
$$\E [f(\underline{Z})]=C \int |\Delta(\underline{z})|^2f(\underline{z})\prod_{\ell=m-1}^{m}(1-|z_\ell|^2)^{n-m-1}
\one_{\{|z_\ell|\le 1\}}(z_\ell) dH_{\mathcal{V}_2}(V_2)\prod_{\ell=m-1}^{m}|dz_\ell|^2.
$$
Thus by consecutive application of Co-area formula  $i$ times, we get
\begin{eqnarray*}
\E [f(\underline{Z})]=C \int |\Delta(\underline{z})|^2f(\underline{z})\prod_{\ell=m-i+1}^{m}(1-|z_\ell|^2)^{n-m-1}
\one_{\{|z_\ell|\le 1\}}(z_\ell) dH_{\mathcal{V}_i}(V_i)\prod_{\ell=m-i+1}^{m}|dz_\ell|^2.
\end{eqnarray*}

Proceeding this way, finally we get

\begin{eqnarray*}
\E f(\underline{Z})=C \int |\Delta(\underline{z})|^2f(\underline{z})\prod_{\ell=1}^{m}(1-|z_\ell|^2)^{n-m-1}
\one_{\{|z_\ell|\le 1\}}(z_\ell) \prod_{\ell=1}^{m}|dz_\ell|^2
\end{eqnarray*}
and this completes the proof.\end{proof}
\begin{proof}[Proof of Theorem \ref{thm:result3}]
For the sake of simplicity, let $n_i=n$ for $i=1,2,\ldots, k$. Let $z_1\geq z_2\geq \cdots \geq z_m$ be the eigenvalues of $A=A_1^{\epsilon_1}A_2^{\epsilon_2}\cdots A_k^{\epsilon_k}$ where $A_i$ be the left uppermost $m\times m$ block of $U_i$ and $U_1, U_2,\ldots,U_k$ be  $n\times n$
 independent Haar distributed unitary matrices, $\epsilon_i=1$ or $-1$. We denote the vector of eigenvalues of $A$ by
 $$\underline Z=(z_1,z_2,\ldots,z_m).$$
Let $f$ be any bounded continuous  function of $\underline{Z}$. In computation of expectation of $f(\underline{Z})$, we approximate Haar measure on direct product of $k$ unitary groups by normalised Lebesgue measure on direct product of their open neighborhoods in $\mathcal{M}(n)$ (using Lemma \ref{lem:haarmeasure}). We apply generalised Schur decomposition to $m\times m$ left uppermost blocks of those $k$ matrices (with powers $\epsilon_i$) and integrate out the unitary matrix variables that come from this generalised Schur decomposition. Then by de-approximating, we get back to integration on direct product of $k$ copies of $\mathcal{V}_{m,n}$. We see that eigenvalues $\underline{Z}$, Jacobian determinant of this schur decomposition are products of diagonal entries of $m\times m$ left uppermost blocks of those $k$ matrices (with powers $\epsilon_i$). So, we would like to integrate out all the variables except the first $m$ diagonal entries of each ${V_i}$. The method of integrating out unwanted variables from each $V_i$ is exactly same as in the proof of \eqref{unitary:specialcase}. We end up with joint probability density of these diagonal variables $\underline{x}$, whose appropriate products are random variables $\underline{Z}$. From there, we get the joint probability density of eigenvalues of $A$.

Coming to the computation,
by Lemma \ref{lem:haarmeasure}, we have
\begin{eqnarray*}
\E[f(\underline{Z})]&=&\int f(\underline{z}) \prod_{i=1}^{k}dH_{\mathcal{U}_i(n)}(U_i)\\
&=&\lim_{\epsilon_i\to 0}\frac{ \int_{\bigcap_{i=1}^k\|X_i^*X_i-I\|<\epsilon_i}f({\underline{z}})
dX_1dX_2\ldots dX_k}{\int_{\bigcap_{i=1}^k \|X_i^*X_i-I\|<\epsilon_i}dX_1dX_2\ldots dX_k}.
\end{eqnarray*}
Limit is taken for all $\epsilon_i$ one by one. For $1\leq i\leq k$, let
$$
X_{i}=
            \left[ {\begin{array}{cc}
             A_{i} & B_{i} \\
             C_{i} & D_{i} \\
                \end{array} } \right].
$$
Now by generalized Schur-decomposition, we have
$$
A_i^{\epsilon_i}=S_iT_i^{\epsilon_i}S_{i+1}^*,\;\;i=1,2,\ldots,k,\; \mbox{and}\;\; k+1=1,
$$
where $T_1,T_2,\ldots,T_k$ are upper triangular matrices
and  $S_1,S_2,\ldots,S_k$ are unitary matrices with $S_{k+1}=S_1$. Let the diagonal entries of $T_i$ be $(x_{i1},x_{i2},\ldots,x_{im})$. We denote $\{x_{ij}:i=1,2,\ldots,k,\;j=1,2,\ldots,m\}$ by $\underline{x}$. Now
$$
 X_{i}=
            \left[ {\begin{array}{cc}
             A_{i} & B_{i} \\
             C_{i} & D_{i} \\
                \end{array} } \right]
               =
            \left[ {\begin{array}{cc}
             S_{i+\frac{1-{\epsilon_i}}{2}} & 0 \\
            0& I \\
                \end{array} } \right]{Y_i}\left[ {\begin{array}{cc}
             S_{i+\frac{1+{\epsilon_i}}{2}}^* & 0 \\
            0& I \\
                \end{array} } \right]
$$
and
$${Y_i}=\left[ {\begin{array}{cc}
            T_{i} & \tilde{B_{i}} \\
             \tilde{C_{i}} & \tilde{D_{i}} \\
                \end{array} } \right].$$
One can see that eigenvalues $z_1,z_2,\ldots,z_m$ of $A=A_1^{\epsilon_1}A_2^{\epsilon_2}\ldots A_k^{\epsilon_k}$ are given by
$$
z_j=\prod_{i=1}^{k}x_{ij}^{\epsilon_i},\;\;j=1,2,\ldots,m.
$$
Now, by using Jacobian determinant calculation for generalised Schur decomposition (see \eqref{eqn:square}),
we get
\begin{eqnarray*}
\prod_{i=1}^{k}|DA_i^{\epsilon_i}|=|\Delta({\underline{z}})|^2\prod_{i=1}^{k}|DT_i^{\epsilon_i}|\prod_{i=1}^{k}|dH(S_i)|,
\end{eqnarray*}
where $\Delta({\underline{z}})=\prod_{i<j}(z_i-z_j)$. Since
$$
|DA_i^{\epsilon_i}|=|\det(A_i)|^{4m(\frac{\epsilon_i-1}{2})}|DA_i|,\
|DT_i^{\epsilon_i}|=|\det(T_i)|^{2(m+1)(\frac{\epsilon_i-1}{2})}|DT_i|$$
and $
|\det(T_i)|=|\det(A_i)|$,
we get
\begin{eqnarray}\label{eqn:van1}
\prod_{i=1}^{k}|DA_i|=|\Delta({\underline{z}})|^2L(\underline{x})\prod_{i=1}^{k}|DT_i||dH(S_i)|,
\end{eqnarray}
where
$$
L(\underline{x})=\prod_{i=1}^k|\det(T_i)|^{(1-\epsilon_i)(m-1)}.
$$
Using (\ref{eqn:van1}) and
Lemma \ref{lem:manifoldmeasure}, we get
\begin{eqnarray}\label{eqn:koushik}
\E[f(\underline{Z})]&=&\lim_{\epsilon_i\to 0}
\frac{ \int_{\bigcap_{i=1}^k\|{Y_i}^*{Y_i}-I\|<\epsilon_i}|\Delta(\underline{z})|^2f({\underline{x}})
L({\underline{x}})
d{Y_1}d{Y_2}\ldots d{Y_k}}{\int_{\bigcap_{i=1}^k \|{Y_i}^*{Y_i}-I \|<\epsilon_i}
|\Delta(\underline{z})|^2L({\underline{x}})d{Y_1}d{Y_2}\ldots d{Y_k}} \nonumber \\
&=&C \int |\Delta(\underline{z})|^2f(\underline{x})L({\underline{x}})\prod_{i=1}^k dH_{\mathcal{V}_i}(V_i)
\end{eqnarray}
where $\mathcal V_i=\mathcal N_{m,n}\cap \mathcal U_i(n)$ and   $C^{-1}=\int |\Delta(\underline{z})|^2L({\underline{x}})\prod_{i=1}^k dH_{\mathcal{V}_i}(V_i)$.
Now for each $\mathcal V_i$, following the arguments given in the  proof of \eqref{unitary:specialcase} (equation \eqref{coarea:1st} onwards), we get that joint probability
 density of $\underline{x}$ is proportional to
$$
\prod_{\ell=1}^{m}\prod_{j=1}^{k}(1-|x_{j\ell}|^2)^{n-m-1}|x_{j\ell}|^{(m-1)(1-\epsilon_j)}\one_{\{|x_{j\ell}|\le 1\}}(x_{j\ell})|dx_{j\ell}|^2\prod_{1\le i<j\le m}|z_i-z_j|^2.
$$
From the above
we get that the joint probability
 density of $\underline{Z}$ is proportional to
$$
\prod_{\ell=1}^{m}\omega(z_{\ell})\prod_{1\le i<j\le m}|z_i-z_j|^2
$$
with a weight function
$$
|dz|^2\omega(z)=\int_{x_1^{\epsilon_1}\cdots x_{k}^{\epsilon_k}
=z}\prod_{j=1}^{k}(1-|x_{j}|^2)^{n-m-1}|x_j|^{(m-1)(1-\epsilon_j)}\one_{\{|x_{j}|\le 1\}}(x_{j})|dx_{j}|^2.
$$
This completes the proof of the theorem when all $n_i$ are equal. If $n_i$ are not all equal, we will have in \eqref{eqn:koushik},
$\mathcal{V}_{i}=\mathcal{N}_{m,n_i}\cap \mathcal{U}(n_i)$. After integrating out all unwanted variables, we will have joint probability
 density of $\underline{Z}$  proportional to
$$
\prod_{\ell=1}^{m}\omega(z_{\ell})\prod_{1\le i<j\le m}|z_i-z_j|^2
$$
with a weight function
$$
|dz|^2\omega(z)=\int_{x_1^{\epsilon_1}\cdots x_{k}^{\epsilon_k}
=z}\prod_{j=1}^{k}(1-|x_{j}|^2)^{n_j-m-1}|x_j|^{(m-1)(1-\epsilon_j)}\one_{\{|x_{j}|\le 1\}}(x_{j})|dx_{j}|^2.
$$
This completes the proof of the theorem.
\end{proof}
It remains to prove Lemma \ref{lem:haarmeasure} and Lemma \ref{lem:manifoldmeasure}.
\begin{proof}[Proof of Lemma \ref{lem:haarmeasure}]
Any $n\times n$ complex matrix $X$ admits QR-decomposition
$$
X=US
$$
where $U$ is unitary matrix, $S$ is  upper triangular matrix with positive real diagonal entries. Then by \eqref{eqn:qrdecomposition},
$$
dX=J(S)|dH_{\mathcal{U}(n)}(U)||DS|
$$
where $J(S)$ is the Jacobian determinant of transformation due to QR-decomposition and is  given by
$$
J(S)=\prod_{i=1}^{n}|S_{i,i}|^{2(n-i+1)-1}.
$$
So we get
$$
\frac{\int_{\|X^*X-I\|<\epsilon} f(X)dX }{\int_{\|X^*X-I\|<\epsilon} dX }=\frac{\int_{\|S^*S-I\|<\epsilon}
 f(US)J(S)dSdH_{\mathcal{U}(n)}(U)} {\int_{\|S^*S-I\|<\epsilon}J(S) dSdH_{\mathcal{U}(n)}(U) }
 $$
Since $f$ is uniformly continuous on the region $\{X:\|X^*X-I\|<\epsilon\}$, given any $r\in\mathbb{N}$,
there exists $\epsilon_r>0$ such that
$$
|f(US)-f(U)|< \frac{1}{2^r} \;\mbox {for all}\; \|S^*S-I\|<\epsilon_r.
$$
Therefore
$$
\Big |\frac{\int_{\|S^*S-I\|<\epsilon_r} f(US)J(S)dSdH_{\mathcal{U}(n)}(U)} {\int_{\|S^*S-I\|<\epsilon_r}J(S)
dSdH_{\mathcal{U}(n)}(U) }-\int f(U)dH_{\mathcal{U}(n)}(U)\Big|<\frac{1}{2^r}
$$
and hence
$$
  \lim_{\epsilon \to 0} \frac{\int_{\|X^*X-I\|<\epsilon}
 f(X)dX }{\int_{\|X^*X-I\|<\epsilon} dX } =\int f(U) dH_{\mathcal{U}(n)}(U).
 $$
\end{proof}
\begin{proof}[Proof of Lemma \ref{lem:manifoldmeasure}]
Let $X\in \mathcal{N}_{m,n}$, then by QR-decomposition, we have
$$
 X=VS,\;\;\;V\in \mathcal{V}:=\mathcal{N}_{m,n}\cap \mathcal U(n),
 $$
where $V$ is $n\times n$ unitary matrix whose $(i,j)$th entry is zero for $1\le j<i\le m$ and
 $S$  upper triangular matrix with positive real diagonal entries. Then by \eqref{qr1}
$$
dX=J_m(S)|dH_{\mathcal{V}}(V)||DS|,
$$
where
$$
J_m(S)=\prod_{i=1}^{m}|S_{i,i}|^{2(n-m)+1}\prod_{i=m+1}^n|S_{i,i}|^{2(n-i)+1}.
$$
Using this decomposition we get
$$
\frac{\int_{\|X^*X-I\|<\epsilon} f(X)dX }{\int_{\|X^*X-I\|<\epsilon} dX }=\frac{\int_{\|S^*S-I\|<\epsilon}
 f(VS)J(S)dSdH_{\mathcal{V}}(V)} {\int_{\|S^*S-I\|<\epsilon}J(S) dSdH_{\mathcal{V}}(V) }.
$$
Since $f$ is uniformly continuous on region $\{X:\|X^*X-I\|<\epsilon\}$, given any $r\in\mathbb{N}$,
there exists $\epsilon_r>0$ such that
$$
|f(VS)-f(V)|< \frac{1}{2^r} \;\mbox {for all}\; \|S^*S-I\|<\epsilon_r.
$$
Therefore
$$
\Big|\frac{\int_{\|S^*S-I\|<\epsilon_r} f(VS)J(S)dSdH_{\mathcal{V}}(V)} {\int_{\|S^*S-I\|<\epsilon_r}J(S)
dSdH_{\mathcal{V}}(V) }-\int f(V)dH_{\mathcal{V}}(V)\Big|<\frac{1}{2^r}
$$
and hence
$$
  \lim_{\epsilon \to 0} \frac{\int_{\|X^*X-I\|<\epsilon}
 f(X)dX }{\int_{\|X^*X-I\|<\epsilon} dX } =\int f(V) dH_{\mathcal{V}}(V).
 $$
\end{proof}

\section{Orthogonal polynomials and kernels}\label{kernel}
In this section we calculate the kernels of the determinantal point process which have appeared in three theorems. First observe that  all weight functions in  three theorems 

are angle independent, that is,  $\omega(z)$ is function
of $|z|$ only. It implies that monic polynomials $P_i(z)=z^i$ are orthogonal with respect to these weight functions.\\

\noindent{\bf Product of Ginibre matrices and inverse of Ginibre matrices:}
Let $p$ be the number of non-inverted Ginibre matrices in Theorem \ref{thm:result1}. Then
\begin{eqnarray*}
\int z^a(\bar{z}^b)\omega(z)|dz|^2
&=&\prod_{j=1}^{k}\int (x_j)^{\epsilon_j a}(\bar{x_j})^{\epsilon_j b}e^{-|x_j|^2}|x_j|^{(1-\epsilon_j)(n-1)}|dx_j|^2
\\&=&\delta_{ab}(2\pi)^k(a!)^p((n-a-1)!)^{k-p}.
\end{eqnarray*}
Corresponding kernel of orthogonal polynomials is given by
\begin{eqnarray*}
\mathbb{K}_n(x,y)=\sqrt{\omega(x)\omega(y)}\sum_{r=0}^{n-1}\frac{(x\bar{y})^r}{(2\pi)^k(r!)^p((n-r-1)!)^{k-p}}.
\end{eqnarray*}

Like in \cite{akemann}, by using Mellin transform, one can  see that weight function $\omega(z)$ can be written as
\begin{eqnarray*}
\omega(z)=(2\pi)^{k-1}G_{k-p,p}^{p,k-p}\l[\begin{array}{c|}(-n,-n,\ldots,-n)_{k-p}\\(0,0,\ldots,0)_p\end{array}\;|z|^2\r]
\end{eqnarray*}
where the symbol $G_{pq}^{nm}(\cdots|z)$ denotes Meijer's G-function. For a detailed discussion on Meijer's G-function, see  \cite{richard}, \cite{Gradshteyn}.\\

\noindent{\bf Product of rectangular matrices:} In Theorem \ref{thm:rectanguular}, for the case of product rectangular matrices
\begin{eqnarray*}
\int z^a(\bar{z}^b)\omega(z)|dz|^2
&=&\prod_{j=1}^{k}\int (x_j)^a(\bar{x_j})^be^{-|x_j|^{2}}|x_j|^{2(n_j-n_1)}|dx_j|^2
\\&=&\delta_{ab}(2\pi)^k\prod_{j=1}^{k}(n_j-n_1+a)!.
\end{eqnarray*}
 Hence, the corresponding kernel of orthogonal polynomials is given by
\begin{eqnarray*}
\mathbb{K}_n(x,y)=\sqrt{\omega(x)\omega(y)}\sum_{r=0}^{n-1}\frac{(x\bar{y})^r}{(2\pi)^k\prod_{j=1}^{k}(n_j-n_1+r)!}.
\end{eqnarray*}
Again using Mellin transform, the weight function $\omega(z)$ can be written as
\begin{eqnarray*}
\omega(z)=(2\pi)^{k-1}G_{0,k}^{k,0}\l[\begin{array}{c|}-\\(n_1-n_1,n_2-n_1,\ldots,n_k-n_1)\end{array}\;|z|^2\r].
\end{eqnarray*}

\noindent{\bf Kernel for product of truncated unitary matrices:} For product of truncated unitary matrices, we calculate the kernel for a particular values of $\epsilon_i$'s.
We assume $\epsilon_i=1$ for $1\le i\le p$ and
$\epsilon_i=-1$ for $p+1\le i\le k$, in Theorem \ref{thm:result3}. Then we have
\begin{eqnarray*}
&&\int z^a(\bar{z}^b)\omega(z)|dz|^2
\\&=&\prod_{j=1}^{k}\int (x_j)^{\epsilon_j a}(\bar{x_j})^{\epsilon_j b}(1-|x_j|^2)^{n_j-m-1}|x_j|^{(m-1)(1-\epsilon_j)}\one_{\{|x_j|\le 1\}}(x_j)|dx_j|^2
\\&=&\delta_{ab}(2\pi)^k\underbrace{\prod_{j=1}^{p}B(a+1, n_j-m)\prod_{j=p+1}^{k}B(m-a, n_j-m)}_{C_a}.
\end{eqnarray*}
 Corresponding kernel of orthogonal polynomials is given by
\begin{eqnarray*}
\mathbb{K}_n(x,y)=\sqrt{\omega(x)\omega(y)}\sum_{r=0}^{n-1}\frac{(x\bar{y})^r}{(2\pi)^k C_r}
\end{eqnarray*}
and the weight function $\omega(z)$ can be written as
$$
(2\pi)^{k-1}\prod_{j=1}^k\Gamma(n_j-m)G_{k,k}^{p,k-p}\l[\begin{array}{c|}(-m,\ldots,-m,n_1-m,\ldots,n_p-m)_{k}\\(0,\ldots,0,-n_{p+1},\ldots,n_k)_k\end{array}\;|z|^2\r].
$$

\section{Limiting spectral distributions}\label{limiting spectral distribution}
In this section we calculate the expected limiting spectral distribution of product of Ginibre and inverse of Ginibre matrices,  product of compatible rectangular matrices and product of truncated unitary matrices.

\begin{theorem}\label{lsd:ginibre}
Let $A_1,A_2,\ldots,A_k$ be independent $n\times n$ random matrices with i.i.d. standard complex Gaussian entries.
Then the limiting expected empirical distribution of square radial part of eigenvalues of
$$
\l(\frac{A_1}{\sqrt{n}}\r)^{\epsilon_1}\l(\frac{A_2}{\sqrt{n}}\r)^{\epsilon_2}\cdots \l(\frac{A_k}{\sqrt{n}}\r)^{\epsilon_k}
$$
where each $\epsilon_i$ is either $1$ or $-1$, is same as the distribution of
$$
U^p\l(\frac{1}{1-U}\r)^{k-p}
$$
where $U$ is a random variable distributed uniformly on $[0,1]$ and $p=\#\{\epsilon_i: \epsilon_i=1\}$. 
\end{theorem}
\begin{proof}
Let $A=A_1^{\epsilon_1}A_2^{\epsilon_2}\cdots A_k^{\epsilon_k}$ and $p=\#\{\epsilon_i: \epsilon_i=1\}$.
We showed in Theorem \ref{thm:result1} that the eigenvalues of $A$ form a determinantal point process  with kernel
$$
\mathbb{K}_n(x,y)=\sqrt{\omega(x)\omega(y)}\sum_{a=0}^{n-1}\frac{(x\bar{y})^a}{(2\pi)^k(a!)^p((n-a-1)!)^{k-p}}
$$
where  $\omega(z)$ is given by
$$
|dz|^2\omega(z)=\int_{x_1^{\epsilon_1}\cdots x_{k}^{\epsilon_k}
=z}e^{\l(-\sum_{j=1}^{k}|x_{j}|^2\r)}
\prod_{j=1}^{k}|x_{j}|^{(1-\epsilon_i)(n-1)}\prod_{j=1}^k|dx_{j}|^2.
$$
Then the scaled one-point corelation function $\frac1n\mathbb{K}_n(z,z)$ gives the density of the expected empirical spectral distribution of $A$ where
$$
\mathbb{K}_n(z,z)=\omega(z)\sum_{a=0}^{n-1}\frac{|z|^{2a}}{(2\pi)^k(a!)^p((n-a-1)!)^{k-p}}.
$$
Let $(X_{n,1},X_{n,2},\ldots,X_{n,k})$ be random variables with joint probability density
$$
\frac{1}{n}e^{-\sum_{j=1}^k|x_j|^{2}}\prod_{j=1}^{k}|x_j|^{(1-\epsilon_j)(n-1)}
\sum_{a=0}^{n-1}\frac{|x_1^{\epsilon_1}x_2^{\epsilon_2}\cdots x_k^{\epsilon_k}|^{2a}}{(2\pi)^k(a!)^p((n-a-1)!)^{k-p}}.
$$
Then $\frac{1}{n}\mathbb{K}_n(z,z)$ is the density of the random variable $X_{n,1}^{\epsilon_1}X_{n,2}^{\epsilon_2}\cdots X_{n,k}^{\epsilon_k}$. Now the density of expected empirical spectral distribution of
$$
\l(\frac{A_1}{\sqrt{n}}\r)^{\epsilon_1}\l(\frac{A_2}{\sqrt{n}}\r)^{\epsilon_2}\cdots \l(\frac{A_k}{\sqrt{n}}\r)^{\epsilon_k}
$$
is the density of random variable
$$
\l(\frac{X_{n,1}}{\sqrt{n}}\r)^{\epsilon_1}\l(\frac{X_{n,2}}{\sqrt{n}}\r)^{\epsilon_2}\cdots \l(\frac{X_{n,k}}{\sqrt{n}}\r)^{\epsilon_k}.
$$
Since the joint probability density of $(X_{n,1},X_{n,2},\ldots,X_{n,k})$ is rotational invariant, so we calculate
the density only for radial part of
$$
\l(\frac{X_{n,1}}{\sqrt{n}}\r)^{\epsilon_1}\l(\frac{X_{n,2}}{\sqrt{n}}\r)^{\epsilon_2}\cdots \l(\frac{X_{n,k}}{\sqrt{n}}\r)^{\epsilon_k}=:Z_n.
$$
Now we have
\begin{eqnarray*}
|Z_n|^2&=&\l(\frac{|X_{n,1}|^2}{n}\r)^{\epsilon_1}\l(\frac{|X_{n,2}|^2}{n}\r)^{\epsilon_2}\cdots \l(\frac{|X_{n,k}|^2}{n}\r)^{\epsilon_k}
\\&=&\l(\frac{R_{n,1}}{n}\r)^{\epsilon_1}\l(\frac{R_{n,2}}{n}\r)^{\epsilon_2}\cdots \l(\frac{R_{n,k}}{n}\r)^{\epsilon_k},\ \mbox{say}.
\end{eqnarray*}
The joint probability density of $(R_{n,1},R_{n,2},\ldots,R_{n,k})$ is proportional to
$$
\frac{1}{n}e^{-\sum_{j=1}^kr_{n,j}}\prod_{j=1}^{k}|r_{n,j}|^{(1-\epsilon_j)(n-1)/2}
\sum_{a=0}^{n-1}\frac{(r_{n,1}^{\epsilon_1}r_{n,2}^{\epsilon_2}\ldots r_{n,k}^{\epsilon_k})^{a}}{(a!)^p((n-a-1)!)^{k-p}}
$$
and the density $f(r)$ of $R_{n,j}$ is given by
$$
f(r)=
\frac{1}{n}e^{-r}\sum_{a=0}^{n-1}\frac{r^a}{a!},\ \ 0< r<\infty.$$
So the density of $\frac{R_{n,j}}{n}$
$$
e^{-nr}\sum_{a=0}^{n-1}\frac{(nr)^{a}}{a!}=\P[\text{Pois}(nr)\le n-1]\to \l\{\begin{array}{lr} 1 & \mbox{if $0<r<1$}
\\0 & \mbox{otherwise}\end{array}\r.
$$
as $n\to \infty$. Hence we have
\begin{eqnarray}\label{eqn:uni1}
\frac{R_{n,j}}{n}\stackrel{\mathcal D}{\rightarrow}U \;\mbox{as}\;n\to \infty
\end{eqnarray}
where $U$ is a random variable distributed uniformly on $[0,1]$.
The joint density of $R_{n,j},R_{n,k}$ is
$$
\frac{1}{n}e^{-(x+y)}\sum_{a=0}^{n-1}\frac{(xy)^a}{a!a!}
$$
if both $\epsilon_j,\epsilon_k$ are either $+1$ or $-1$. Then
\begin{eqnarray*}
&&\E[|R_{n,j}-R_{n,k}|^2]
\\&=&\int_{0}^{\infty}\int_{0}^{\infty}\frac{1}{n}e^{-(x+y)}(x-y)^2\sum_{a=0}^{n-1}\frac{(xy)^a}{a!}dxdy\\
&=&\frac{2}{n}\int_{0}^{\infty}\int_{0}^{\infty}e^{-(x+y)}\sum_{a=0}^{n-1}\frac{x^{a+2}y^a}{a!a!}dxdy
-\frac{2}{n}\int_{0}^{\infty}\int_{0}^{\infty}e^{-(x+y)}\sum_{a=0}^{n-1}\frac{x^{a+1}y^{a+1}}{a!a!}dxdy
\\&=&\frac{2}{n}\l[\sum_{a=0}^{n-1}\{(a+2)(a+1)-(a+1)^2\}\r]
\\&=&n+1.
\end{eqnarray*}
Therefore
\begin{eqnarray}\label{eq:sub}
\l(\frac{R_{n,j}}{n}-\frac{R_{n,k}}{n}\r) \stackrel{L^2}{\to} 0\;\mbox{ as }n\to \infty.
\end{eqnarray}
If $\epsilon_j=1,\epsilon_k=-1$, then the joint density of $R_{n,j},R_{n,k}$ is
$$
\frac{1}{n}e^{-(x+y)}\sum_{a=0}^{n-1}\frac{x^ay^{n-1-a}}{a!(n-1-a)!}.
$$
Therefore we have
\begin{eqnarray*}
&&\E[|R_{n,j}+R_{n,k}-n|^2]\\&=&\frac{1}{n}\int_{0}^{\infty}\int_{0}^{\infty}e^{-(x+y)}(x+y-n)^2\sum_{a=0}^{n-1}
\frac{x^ay^{(n-a-1)}}{a!(n-a-1)!}dxdy
\\&=&\frac{2}{n}\int_{0}^{\infty}e^{-x}\sum_{a=0}^{n-1}\frac{x^{a+2}}{a!}dx
-4\int_{0}^{\infty}e^{-x}\sum_{a=0}^{n-1}\frac{x^{a+1}}{a!}dx
\\&&+\frac{2}{n}\int_{0}^{\infty}\int_{0}^{\infty}e^{-(x+y)}\sum_{a=0}^{n-1}\frac{x^{a+1}y^{n-a}}{a!(n-a-1)!}dxdy+n^2
\\&=&\frac{2}{n}\sum_{a=0}^{n-1}(a+2)(a+1)-4\sum_{a=0}^{n-1}(a+1)+\frac{2}{n}\sum_{a=0}^{n-1}(a+1)(n-a)+n^2
\\&=&\frac{2}{n}\sum_{a=0}^{n-1}(n+2)(a+1)-4\sum_{a=0}^{n-1}(a+1)+n^2
\\&=&(n+2)(n+1)-2n(n+1)+n^2=n+2
\end{eqnarray*}
and hence
\begin{eqnarray}\label{eq:inv}
\l(\frac{R_{n,j}}{n}+\frac{R_{n,k}}{n}-1\r) \stackrel{L^2}{\to} 0\;\mbox{ as }n\to \infty.
\end{eqnarray}
Now combining (\ref{eqn:uni1}), (\ref{eq:sub}) and (\ref{eq:inv}), we get
$$
\l(\frac{R_{n,1}}{n}\r)^{\epsilon_1}\l(\frac{R_{n,2}}{n}\r)^{\epsilon_2}
\cdots\l(\frac{R_{n,k}}{n}\r)^{\epsilon_k}\stackrel{\mathcal D}{\to}U^p\l[\frac{1}{1-U}\r]^{k-p}.
$$
where $p=\#\{\epsilon_i:\epsilon_i=1\}$.
\end{proof}
\begin{remark}
If $k=1$ and $\epsilon=1$, then it follows from Theorem \ref{lsd:ginibre} that the expected limiting spectral distribution of properly scaled Ginibre matrix is well known \textit{circular} law. If $k=2$ with $\epsilon_1=-1$, $\epsilon_2=1$, we get the expected limiting spectral distribution of  \textit{spherial ensemble}.
\end{remark}
In the following theorem we describe the limiting distribution of the radial part of the eigenvalues of product of rectangular matrices.
\begin{theorem}\label{lsd:rectangular}
Let $A_1,A_2,\ldots,A_k$ be independent rectangular matrices of dimension $n_i\times n_{i+1}$ for $i= 1,2,\ldots,k$,
with $n_{k+1}=n_1=\min\{n_1,n_2,\ldots,n_k\}$, and with i.i.d. standard complex Gaussian entries. If $\frac{n_j}{n_1}\to \alpha_j$ as $n_1\to \infty$ for $j=2,3,\ldots,k$, then the limiting expected empirical distribution of square radial part of eigenvalues of $\frac{A_1}{\sqrt{n}}\frac{A_2}{\sqrt{n}}\cdots \frac{A_k}{\sqrt{n}}$ is same  as the distribution  of following random variable
$$
U(U-1+\alpha_2)\cdots(U-1+\alpha_{k})
$$
where $U$ is a uniform random variable on $[0,1]$.
\end{theorem}
\begin{proof}
We have shown in Theorem \ref{thm:rectanguular} that the eigenvalues of $A_1A_2\ldots A_k$ form a determinantal point process with kernel
$$
\mathbb{K}_n(x,y)=\sqrt{\omega(x)\omega(y)}\sum_{a=0}^{n-1}\frac{(x\bar{y})^a}{(2\pi)^k\prod_{j=1}^{k}(n_j-n_1+a)!}
$$
where $\omega(z)$ is the weight function, given by
$$
|dz|^2\omega(z)=\int_{x_1x_2\ldots x_k=z}e^{-\sum_{j=1}^{k}|x_j|^2}\prod_{j=1}^{k}|x_j|^{2(n_j-n_1)}\prod_{j=1}^k|dx_j|^2.
$$
Then the scaled one-point corelation function $\frac1n\mathbb{K}_n(z,z)$ gives the density of the expected empirical spectral distribution of $A_1A_2\cdots A_k$. Let $n_1=n$ and
$(X_{n,1},X_{n,2},\ldots,X_{n,k})$ be random variables with joint probability density
$$
\frac{1}{n}e^{-\sum_{j=1}^k |x_j|^{2}}\prod_{j=1}^{k}|x_j|^{2(n_j-n)}
\sum_{a=0}^{n-1}\frac{|x_1x_2\ldots x_k|^{2a}}{(2\pi)^k\prod_{j=1}^{k}(n_j-n+a)!}.
$$
Then the density of $(X_{n,1}X_{n,2}\cdots X_{n,k})$ is given by $\frac{1}{n}\mathbb K_n(z,z)$. Now the density of expected empirical spectral distribution of
$\frac{A_1}{\sqrt{n}}\frac{A_2}{\sqrt{n}}\cdots \frac{A_k}{\sqrt{n}}$
is same as the density of the random variable
$Z_n=\frac{X_{n,1}}{\sqrt{n}}\frac{X_{n,2}}{\sqrt{n}}\cdots \frac{X_{n,k}}{\sqrt{n}}.$
Clearly, the joint probability density of $(X_{n,1},X_{n,2},\ldots,X_{n,k})$ is rotational invariant. So we calculate
the density of square of radial part of $Z_n$. We have
$$
|Z_n|^2=\frac{|X_{n,1}|^2}{n}\frac{|X_{n,2}|^2}{n}\cdots\frac{|X_{n,k}|^2}{n}
=\frac{R_{n,1}}{n}\frac{R_{n,2}}{n}\cdots\frac{R_{n,k}}{n},\;\;\mbox{say}.
$$
The joint probability density of $(R_{n,1},R_{n,2},\ldots,R_{n,k})$ is
$$
\frac{1}{n}e^{-\sum_{j=1}^{k}r_j}\prod_{j=1}^{k}r_j^{(n_j-n)}\sum_{a=0}^{n-1}\frac{(r_1r_2\ldots r_k)^a}{\prod_{j=1}^{k}(n_j-n+a)!}.
$$
Now by routine calculation it can be shown that
\begin{eqnarray}
&&\frac{R_{n,1}}{n}\stackrel{\mathcal D}{\to}U \;\;\mbox{as $n\to \infty$},\label{eqn:uni}
\\&&\E(R_{n,1}-R_{n,j})=(n-n_j)\;\;\mbox{for $j=2,3,\ldots,k$},\label{eqn:exp}
\\&&\E[(R_{n,1}-R_{n,j})^2]=(n-n_j)^2+n_j+1\;\mbox{for $j=2,3,\ldots,k$},\label{eqn:var}
\end{eqnarray}
where $U$ is a uniform random variable on $[0,1]$. By (\ref{eqn:exp}) and (\ref{eqn:var}), we have
\begin{eqnarray}\label{eqn:L2}
\frac{R_{n,1}}{n}-\frac{R_{n,j}}{n}-(1-\alpha_j)\stackrel{L^2}{\to}0\;\;\mbox{for $j=2,3,\ldots,k$}.
\end{eqnarray}
Therefore by (\ref{eqn:uni}) and (\ref{eqn:L2}), we conclude that the limiting distribution of $|Z_n|^2$ is same as the distribution of the following random variable
$$
U(U-1+\alpha_2)\cdots(U-1+\alpha_k).
$$
This completes proof of the theorem.
\end{proof}
In the following theorem we describe the limiting distribution of radial part of eigenvalues of product of truncated unitary matrices.
\begin{theorem}\label{lsd:unitary}
Let $U_1,U_2,\ldots,U_k$ be $k$ independent Haar distributed unitary matrices of dimension $n_i\times n_{i}$  for $i= 1,2,\ldots,k$ respectively and $A_1,A_2,\ldots,A_k$ be $m\times m$ left uppermost blocks of $U_1,U_2,\ldots,U_k$ respectively. If $\frac{n_i}{m}\to \alpha_i$ as $m\to \infty$ for $i=1,2,\ldots,k$, then the limiting expected empirical distribution of square radial part of eigenvalues of $A_1^{\epsilon_1}A_2^{\epsilon_2}\ldots A_k^{\epsilon_k}$ is same  as the distribution of following random variable
$$
\prod_{i=1}^{k}\left(\frac{\frac{1-\epsilon_i}{2}+\epsilon_iU}{\alpha_i-\frac{1+\epsilon_i}{2}+\epsilon_iU}\right)^{\epsilon_i}
$$
where $U$ is a random variable uniformly distributed on $[0,1]$ and each $\epsilon_i$ is $+1$ or $-1$.
\end{theorem}
\begin{proof}
We have shown that the eigenvalues of $A_1^{\epsilon_1}A_2^{\epsilon_2}\ldots A_k^{\epsilon_k}$ form a determinantal point process with kernel
$$
\mathbb{K}_m(x,y)=\sqrt{\omega(x)\omega(y)}\sum_{a=0}^{m-1}\frac{(x\bar{y})^a}{(2\pi)^kC_a}
$$
where $C_a=\prod_{\{j:\epsilon_j=1\}}B(a+1,n_j-m)\prod_{\{j:\epsilon_j=-1\}}B(m-a,n_j-m)$ and $\omega(z)$ is the weight function, given by
$$
|dz|^2\omega(z)=\int_{x_1^{\epsilon_1}x_2^{\epsilon_2}\ldots x_k^{\epsilon_k}=z}\prod_{j=1}^{k}(1-|x_j|^2)^{(n_j-m-1)}|x_j|^{(1-\epsilon_j)(m-1)}\one_{|x_j|\le 1}\prod_{j=1}^k|dx_j|^2.
$$
Then the density of expected empirical spectral distribution of $A_1^{\epsilon_1}A_2^{\epsilon_2}\ldots A_k^{\epsilon_k}$ is given by $\frac1m \mathbb K_m(z,z)$.
Let $(X_{m,1},X_{m,2},\ldots,X_{m,k})$ be random variables with joint probability density
$$
\frac{1}{m}\prod_{j=1}^{k}(1-|x_j|^2)^{(n_j-m-1)}|x_j|^{(1-\epsilon_j)(m-1)}\one_{|x_j|\le 1}\sum_{a=0}^{m-1}\frac{|x_1^{\epsilon_1}x_2^{\epsilon_2}\ldots x_k^{\epsilon_k}|^{2a}}{(2\pi)^kC_a}.
$$
Then it is easy to see that the density of $Z_m=X_{m,1}^{\epsilon_1}X_{m,2}^{\epsilon_2}\cdots X_{m,k}^{\epsilon_k}$ is also $\frac{1}{m}\mathbb K_m(z,z)$. Clearly, the joint probability density of $(X_{m,1},X_{m,2},\ldots,X_{m,k})$ is rotational invariant. So we calculate
the density for square radial part of $Z_m$. We have
$$
|Z_m|^2=|X_{m,1}|^{2\epsilon_1}|X_{m,2}|^{2\epsilon_2}\cdots|X_{m,k}|^{2\epsilon_k}
=R_{m,1}^{\epsilon_1}R_{m,2}^{\epsilon_2}\cdots R_{m,k}^{\epsilon_k},\;\mbox{say}.
$$
Now the joint probability density of $(R_{m,1},R_{m,2},\ldots,R_{m,k})$ is
$$
\frac{1}{m}\prod_{j=1}^{k}(1-r_j)^{(n_j-m-1)}r_j^{\frac{(1-\epsilon_j)}{2}(m-1)}\one_{0<r_j\le 1}\sum_{a=0}^{m-1}\frac{|r_1^{\epsilon_1}r_2^{\epsilon_2}\ldots r_k^{\epsilon_k}|^{a}}{C_a}.
$$
For $\epsilon_i=1$, density of $R_{m,i}$ for $i=1,2,\ldots,m$ is given by
$$
\frac{1}{m}\sum_{a=0}^{m-1}\frac{(n_i-m+a)!}{a!(n_i-m-1)!}(1-r)^{(n_i-m-1)}r^a.
$$
Therefore for any $\ell\in \mathbb{N}$, we have
\begin{eqnarray}\label{eqn:+1}
\E[R_{m,i}^{\ell}]&=&\int_{0}^{\infty}\frac{1}{m}\sum_{a=0}^{m-1}\frac{(n_i-m+a)!}{a!(n_i-m-1)!}
(1-r)^{(n_i-m-1)}r^{\ell+a}\nonumber
\\&=&\frac{1}{m}\sum_{a=0}^{m-1}\frac{(n_i-m+a)!}{a!(n_i-m-1)!}\frac{(\ell+a)!(n-m-1)!}{(n-m+\ell+r)!}\nonumber
\\&=&\frac{1}{m}\sum_{a=0}^{m-1}
\frac{(a+\ell)(a+\ell-1)\cdots(r+1)}{(n_i-m+a+\ell)(n_i-m+a+\ell-1)\cdots(n_i-m+a+1)}\nonumber
\\&{\longrightarrow}&\int_{0}^{1}\frac{x^{\ell}}{(\alpha_i-1+x)^{\ell}}dx\;\;\;\mbox{as $m\to \infty$}\nonumber
\\&=&\E\l[\l(\frac{U}{\alpha_i-1+U}\r)^{\ell}\r],
\end{eqnarray}
where $U$ is uniform random variable on $[0,1]$. By similar way it can be shown that  if $\epsilon_i=-1$, then
for any $\ell\in\mathbb{N}$,
\begin{eqnarray}\label{eqn:-1}
\E[R_{m,i}^{\ell}]\to \E\l[\l(\frac{1-U}{\alpha_i-U}\r)^{\ell}\r] \;\;\mbox{as $m\to \infty$}.
\end{eqnarray}

If $\epsilon_i=1$ and $\epsilon_j=1$, then it is not hard to see that
\begin{eqnarray}\label{eqn:+1+1}
\E\l[\frac{(\alpha_i-1)R_{m,i}}{1-R_{m,i}}-\frac{(\alpha_j-1)R_{m,j}}{1-R_{m,j}}\r]^2\to 0\;\;\mbox{as $m\to \infty$},
\end{eqnarray}
and if $\epsilon_i=+1$ and $\epsilon_j=-1$, then
\begin{eqnarray}\label{eqn:+1-1}
\E\l[\frac{(\alpha_i-1)R_{m,i}}{1-R_{m,i}}+\frac{(\alpha_j-1)R_{m,j}}{1-R_{m,j}}-1\r]^2\to 0\;\;\mbox{as $m\to \infty$}.
\end{eqnarray}
Now by (\ref{eqn:+1}), (\ref{eqn:-1}), (\ref{eqn:+1+1}) and (\ref{eqn:+1-1}), it follows that  the limiting  distribution of $|Z_m|^2$ is same as the distribution of the random variable
$$
\prod_{i=1}^{k}\left(\frac{\frac{1-\epsilon_i}{2}+\epsilon_iU}{\alpha_i-\frac{1+\epsilon_i}{2}+\epsilon_iU}\right)^{\epsilon_i}
$$

and this completes the proof.
\end{proof}

\section{Appendix}\label{appendix}
In this section we provide details of Jacobian determinant calculations of the proof of Theorem \ref{thm:result1} and Theorem \ref{thm:rectanguular}, and some discussion on techniques used in Theorem \ref{thm:result3} using manifolds.
Before doing the Jacobian determinant calculation, we state a few basic facts about wedge product.

If $dy_j=\sum_{k=1}^na_{j,k}dx_k$, for $1\le j\le n$, then using the alternating property $dx\wedge dy=-dy\wedge dx$ it is easy to see that
\begin{equation}\label{eqn:wedge:relation}
dy_1\wedge dy_2\wedge \ldots\wedge dy_n=\\det(a_{j,k})_{j,k\le n}dx_1\wedge x_2\wedge\ldots\wedge dx_n.
\end{equation}

\noindent{\bf Schur decomposition and a change of measure:} Any matrix $M\in g\ell(n,\mathbb{C})$ can be written as
\begin{equation}\label{schur:decomposition}
M=V(Z+T)V^*
\end{equation}
where $V$ is  unitary matrix, $T$ is  strictly upper triangular matrix and $Z$ is  diagonal matrix.  This decomposition is not unique, but uniqueness can be restored by
ordering $z_1,z_2,\ldots,z_n$ and then imposing that eigenvalues are distict and $V_{i,i}> 0$ for $i=1,2,\ldots,n$. For detailed discussion on this see (6.3.1) and thereof  on page 103 of \cite{manjubook}. We shall consider this decomposition with the condition that $V_{i,i}\geq 0$ for $i=1,2,\ldots,n$.

Integration of a function of $M$ with respect to  Lebesgue measure  is the same as integrating against the $2n^2$-form
$$2^{n^2}i^{n^2}\bigwedge_{i,j}(dM_{i,j}\wedge d\bar{M}_{i,j}).$$
We shall write the Lebesgue measure on $M$ in terms of $Z,V,T$. For this we need the Jacobian determinant formula for the change of variables from $\{dM_{i,j},d\bar M_{i,j}\}$ to $dz_i,d\bar{z}_i,\ 1\leq i\leq n;\ dT_{i,j},d\bar{T}_{i,j}, i<j,$ and $\Omega$ where $\Omega=(\omega_{i,j}):=V^*dV$.  The  Jacobian determinant formula is given by
\begin{equation}\label{jacobian:manjubook}
\bigwedge_{i,j}(dM_{i,j}\wedge d\bar{M}_{i,j})=\prod_{i<j}|z_i-z_j|^2\bigwedge_{i<j}|\omega_{i,j}|^2\bigwedge_{i}|dz_{i}|^2
\bigwedge_{i<j}|dT_{i,j}|^2,
\end{equation}
Here we have written $|\omega|^2$ for $\omega\wedge \bar \omega$. For the proof of \eqref{jacobian:manjubook} see (page 104) \cite{manjubook}.\\

\noindent{\bf RQ-decomposition of rectangular matrices and a change of measure:} 

For Jacobian determinant calculation of RQ-decompostion of rectangular matrices in real case, see (chapter 2) \cite{muirhead}. Result for complex case is discussed in Edelman's work \cite{edelman}. Since this result is needed in Jacobian determinant calculation of Generalised Schur decomposition , for the sake of completeness, we present its proof here.

   Any $m\times n$ complex matrix $M$ with $m\le n$ can be written as
\begin{eqnarray}\label{eqnRQ}
M=SU^*
\end{eqnarray}
where $S$ is a $m\times m$ upper triangular matrix and $U^*$ has orthonormal rows with non-negative real
diagonal entries. This can be done by applying Gram-Schmidt orthogonalization process to the rows of $M$ from
 bottom to top and fixing the argument of diagonal entries of $U^*$ to be zero.
The decomposition is not unique if $M$ is not of full rank or $U^*$ has at least one zero on diagonal. But we shall
omit all such $M$ matrices
(a lower dimensional set and hence also of zero Lebesgue measure).
From (\ref{eqnRQ}) we get
$$
dM=S(dU^*)+(dS)U^* .
$$
Let $V$ be such that $[U \;\;V]$ is  $n \times n$ unitary matrix.
\begin{eqnarray}\label{eq:RQ1}
\Lambda:=(dM)[U \;\;V]&=&(S(dU^*)+(dS)U^*)[U\;\; V]\nonumber
\\&=&S(dU^*)[U\;\;V]+dS[I\;\;0]\nonumber
\\&=&S\Omega+[dS\;\;0],
\end{eqnarray}
where $\Omega:=(dU^*)[U\;\;V]=(\omega_{i,j})$ and $\Lambda=(\lambda_{i,j})$ are $m\times n$ matrices of one forms. Also observe that, the leftmost $m\times m$ block of $\Omega$ is skew-Hermitian.

Integration of a function of $M$ with respect to Lebesgue measure is same as integrating against the $2nm$-form
$$
\bigwedge_{i,j}(dM_{i,j}\wedge d\bar{M}_{i,j}).
$$
There should be a factor of $2^{mn}i^{mn}$, but to make life simple, we shall omit constants  in all Jacobian determinant calculations. Where probability
measures are involved, these constants can be recovered at the
end by finding normalizing constants.

Now we want to write the Lebesgue measure on $M$ in terms of $S$ and $U$.
For this we must find the  Jacobian determinant for the change of variables from $\{dM_{i,j}, d\bar{M}_{i,j},\ 1\leq i\leq m,1\leq j \leq n \}$
to $\{dS_{i,j},\ 1\leq i,j \leq m\}$ and $\Omega$. Since for any fixed unitary matrix $W$, the transformation  $M\to MW$ is unitary
 on the set of $m\times n$ complex matrices, we have
\begin{equation}\label{M:lambda}
\bigwedge_{i,j}(dM_{i,j}\wedge d\bar{M}_{i,j})=\bigwedge_{i,j}(\lambda_{i,j}\wedge \bar{\lambda}_{i,j}).
\end{equation}
Thus we just have to find the Jacobian determinant for the change of variables from $\Lambda$ to $\Omega,dS$ and their conjugates. We write (\ref{eq:RQ1}) in the following way
\begin{eqnarray}\label{lambda:transformation}
\lambda_{i,j}&=&dS_{i,j}+\sum_{k=1}^mS_{i,k}\omega_{k,j} \nonumber\\
&=&\left\{
\begin{array}{lcr}
\displaystyle{S_{i,i}\omega_{i,j}+\left[\sum_{k=i+1}^{m}S_{i,k}\omega_{k,j}\right]}&\mbox{if }&j<i\le m.\\
& &\\
\displaystyle{dS_{i,i}+S_{i,i}\omega_{i,i}+\left[\sum_{k=i+1}^{m}S_{i,k}\omega_{k,j}\right]}&\mbox{if }&i=j.\\
& &\\
\displaystyle{dS_{i,j}+\left[\sum_{k=i}^{m}S_{i,k}\omega_{k,j}\right]}&\mbox{if }& i<j\le m.\\
& &\\
\displaystyle{S_{i,i}\omega_{i,j}+\left[\sum_{k=i+1}^{m}S_{i,k}\omega_{k,j}\right]}&\mbox{if }& j>m.
\end{array}
\right.
\end{eqnarray}
Now we arrange  $\{\lambda_{i,j}, \bar{\lambda}_{i,j}\}$
in the ascending  order given by the following relation
$$(i,j)\leq(r,s)\;\mbox{if}\; i>r \; \mbox{or if }\; i=r\; \mbox{and}\;j\leq s.$$
Also observe that the expressions inside square brackets in \eqref{lambda:transformation} involve only those one-forms
that have already appeared before in the given ordering of one-forms $\{\lambda_{i,j}, \bar{\lambda}_{i,j}\}$. Recall that the leftmost $m\times m$ block of $\Omega$ is skew-Hermitian, that is, $\omega_{i,j}=-\bar{\omega}_{j,i}$
for $i,j\le m$. Now taking wedge products of $\lambda_{i,j}$ in the above mentioned order and  using the transformation rules given in \eqref{lambda:transformation}, and with the help of last two observations,  we get that
\begin{eqnarray}\label{eqn:RQ2}
\bigwedge_{i,j}|\lambda_{i,j}|^2&=&
\prod_{i=1}^m|S_{i,i}|^{2(n-m+i-1)}\bigwedge_{i}|dS_{i,i}+S_{i,i}\omega_{i,i}|^2\bigwedge_{i<j}|dS_{i,j}|^2
\bigwedge_{i<j}|\omega_{i,j}|^2
\nonumber\\&=&\prod_{i=1}^m|S_{i,i}|^{2(n-m+i-1)}\bigwedge_{i}|dS_{i,i}|^2\bigwedge_{i<j}|dS_{i,j}|^2
\bigwedge_{i<j}|\omega_{i,j}|^2.
\end{eqnarray}
We arrive at the last step in \eqref{eqn:RQ2}  by observing that $\omega_{k,k}\wedge_{i<j}|\omega_{i,j}|^2=0$ for any $k\le m$,
because $\{U_{n \times m}:U^*U=I,U_{i,i}>0\}$ is a smooth manifold of dimension $(2nm-m^2-m)$ and  its complement in
$\{U_{n \times m}:U^*U=I,U_{i,i}\ge 0\}$ is of measure zero and $\omega_{k,k}\wedge_{i<j}|\omega_{i,j}|^2$ is
an $(2nm-m^2-m+1)$-form.

Finally,
using \eqref{M:lambda} and \eqref{eqn:RQ2} we arrive at the following Jacobian determinant formula
\begin{equation}\label{M:jacobian}
\bigwedge_{i,j}|dM_{i,j}|^2=\prod_{i=1}^m|S_{i,i}|^{2(n-m+i)-1}\bigwedge_{i\leq j}|dS_{i,j}|^2
\bigwedge_{i<j}|\omega_{i,j}|^2.
\end{equation}
This Jacobian determinant formula will be main ingredient in the Jacobian determinant calculation of product of rectangular matrices.

In subsequent paragraphs, we shall denote
\begin{equation}\label{eqn:DS}\bigwedge_{i<j}|\omega_{i,j}|^2=|dH(U)| \ \mbox{and}\ \bigwedge_{i\le j}|dS_{i,j}|^2=|DS|.\end{equation}

\noindent{\bf Jacobian determinant calculation for product of rectangular matrices:}
  
  Akemann and Burda \cite{akemann} have calculated Jacobian determinant for Generalised Schur decomposition of square matrices. Here, we generalise that result to the case of rectangular matrices.
 
 Let $A_1,A_2,\ldots,A_k$ be $k$ rectangular matrices. Size of $A_i$ is $n_i\times n_{i+1}$ for
$i=1,2,\ldots,k$  and   $n_{k+1}=n_1$, $n_1=\min\{n_1,n_2,\ldots,n_k\}$. By Schur-decomposition
$$
A_1A_2\ldots A_k=U_1TU_1^*
$$
where $U_1$ is $n_1\times n_1$ unitary matrix with non negative real diagonal entries and $T$
is upper triangular matrix with diagonal entries in descending
lexicographic ordering (lexicographic order: $u+iv\leq u'+iv'$ if $u<u'$ or if $u=u'$ and $v\leq v'$).

Now by sequential application of RQ-decomposition starting from $i=1$, $U_i^*A_i=S_iU_{i+1}^*$
for $i=1,2,\ldots,k-1$ where $S_i$ is $n_1\times n_1$ upper triangular matrix , $U_{i+1}^*$ is
$n_1\times n_{i+1}$ matrix with orthonormal rows and with non-negative real diagonal entries. Let $V_i$ be such that  $[U_i\;\;V_i]$
be a $n_i\times n_i$ unitary matrix for $i=2,3,\ldots,k$. Here $V_i$ is constructed as follows. Let $e_1,e_2,\ldots,e_{n_i}$ be standard basis for $\mathbb{C}^{n_i}$. $e_{i_1}$ be the basis vector with least index, not belonging to the span of (columns of $U_i$). The unit vector gotten by normalizing the projection of $e_{i_1}$ onto orthogonal complement of (columns of $U_i$) is chosen as first column of $V_i$. Let $e_{i_2}$ be the basis vector with least index, not belonging to the span of (columns of $U_i$, first column of $V_i$). The unit vector gotten by normalizing the projection of $e_{i_2}$ onto orthogonal complement of (columns of $U_i$, first column of $V_i$) is chosen as second column of $V_i$. Proceeding this way, we get required $V_i$.

We want to write the Lebesgue  measure on $(A_1,A_2,\ldots,A_k)$ in terms of
$U_1,\ldots,U_k$, $S_1,S_2,\ldots,S_{k-1},$ $T$ and $V_2^*A_2,V_3^*A_3,\ldots,V_k^*A_k$. On the complement
of a set of measure zero, the above decompositions are unique. So we shall omit all collections of
matrices $A_1,A_2,\ldots,A_k$ which are in this set of measure zero
and then $T,S_1,S_2,\ldots,S_{k-1}$ are invertible, all possible products of $A_1,A_2,\ldots,A_k$
are of full rank and eigenvalues of $T$ are distinct. We apply the following transformations step by step to
arrive at the measure written in terms
of new variables.

\noindent{\bf Step 1:}  We first transform
$$(A_1,A_2,\ldots,A_k ) \to (X_1,X_2,\ldots,X_k)$$
where $X_i=A_i$ for $i=1,2,\ldots,k-1$ and
$$
X_k=\l[\begin{array}{c}
    A_1A_2\ldots A_{k-1}\\
    V_k^*
    \end{array}\r]A_k
$$
where $V_k^*$ is $(n_k-n_1)\times n_k$ matrix with orthonormal rows. Also the rows of $V_k^*$ are orthogonal to rows of $A_1A_2\ldots A_{k-1}$.
It is easy to see that the Jacobian determinant formula  for this transformation is given by
\begin{equation}\label{eqn:step1}
\bigwedge_{i=1}^k |DA_i|=\\det((A_1A_2\ldots A_{k-1})(A_1A_2\ldots A_{k-1})^*)^{-n_1}\bigwedge_{i=1}^k|DX_i|.
\end{equation}
\noindent{\bf Step 2:}
By applying Schur-decomposition to upper $n_1\times n_1$ block of $X_k$ we get
$$
X_k=\l[\begin{array}{c}
    U_1TU_1^*\\
    B_k
    \end{array}\r],
$$
where $B_k=V_k^* A_k$. Using \eqref{jacobian:manjubook}, the Lebesgue measure on $X_k$ can be  written in terms of $U_1,T,B_k$ as follows
\begin{equation}\label{eqn:step2}
|DX_k|=|\Delta(T)|^2|dH(U_1)||DT||DB_k|,
\end{equation}
where $|dH(U_1)|$ is Haar measure on $\mathcal{U}(n_1)/\mathcal{U}(1)^{n_1}$, $|DT|$ is the Lebesgue measure on T and
 $$
 \Delta(T)=\prod_{1\le i<j\le n_1}(T_{i,i}-T_{j,j}).
 $$
 If we denote the eigenvalues of $A_1A_2\ldots A_k$ by $z_1,z_2,\ldots,z_{n_1}$, then $\Delta(T)$ is basically equal to
 $$
 \Delta(T)=\prod_{1\le i<j\le n_1}(T_{i,i}-T_{j,j})=\prod_{1\le i<j\le n_1}(z_i-z_j).
 $$
\noindent{\bf Step 3:} Now we apply the following transformation
$$
X_1\to U_1^*X_1= S_1U_2^*
$$
where $U_1$ is as in Step 2 and second part of above equation is by RQ-decomposition of
$U_1^*X_1$. $U_2^*$ is $n_1\times n_2$ matrix with orthonormal rows and non-negative real
diagonal entries and  $S_1$ is $n_1\times n_1$ upper triangular matrix. We shall omit matrices $X_1$
for which $U_1^*X_1$ is not of full rank (this set is of measure zero). Now using \eqref{M:jacobian},
the Lebesgue measure on $X_1$ can be written  in terms of $U_2,S_1$ as follows
\begin{equation}\label{eqn:step3}
|DX_1|=J(S_1) |dH(U_2)||DS_1|
\end{equation}
where
$$
J(S_1)=\prod_{i=1}^{n_1}|S_1({i,i})|^{2(n_2-n_1+i-1)}.
$$

\noindent{\bf Step $i+2$ for $i=2,3,\ldots, k-1$:} At $(i+2)$-th step we apply the following transformation
 \begin{equation}\label{transformation_i+2}
X_i\to\l[\begin{array}{c}
    U_i^*\\
    V_i^*
    \end{array}\r]X_i=
\l[\begin{array}{c}
    S_iU_{i+1}^*\\
    B_i
    \end{array}\r],
\end{equation}
where $U_i$ is as in Step $i+1$ and  $[U_i\;\;V_i]$ is an unitary matrix.  The  second part of the above
equation is obtained by  RQ-decomposition of $U_i^*X_i$, where $U_{i+1}^*$ is $n_1\times n_{i+1}$ matrix with
orthonormal rows and non-negative real diagonal entries, and $S_i$ is $n_1\times n_1$ upper triangular matrix. Also note that $B_i=V_i^*X_i=V_i^*A_i$ for $2\leq i \leq k-1$.
We shall omit matrices $X_i$ for which $U_i^*X_i$ is not of full rank (this set is of measure zero).
Now using \eqref{M:jacobian} the Lebesgue measure on $X_i$ can be written in terms of $U_{i+1},S_i,B_i$ as
\begin{equation}\label{dsf:DS_i}
|DX_i|=J(S_i)|dH(U_{i+1})||DS_i||DB_i|
\end{equation}
where
$$
J(S_i)=\prod_{j=1}^{n_1}|S_i({j,j})|^{2(n_{i+1}-n_1+j-1)}.
$$

\noindent{\bf Step k+2:} Now we transform $T$ to $S_k$ as follows
$$
T\to S_k:=(S_1S_2\ldots S_{k-1})^{-1}T.
$$
The Jacobian determinant formula for this transformation is given by
\begin{equation}\label{eqn:stepk+2}|DT|=\prod_{i=1}^{k-1}L(S_i)|DS_k|,\end{equation} where
$$
L(S_i)=\prod_{j=1}^{n_1}|S_i(j,j)|^{2(n_1-j+1)}.
$$

Applying the above transformations in the given order, we can write Lebesgue  measure on
 $(A_1,A_2,\ldots,A_k)$ in terms of $U_1,U_2,\ldots,U_k,S_1,S_2,\ldots,S_{k},B_2,B_3,\ldots,B_k$.
  One can observe that
$$
A_1A_2\ldots A_{k-1}=U_1S_1S_2\ldots S_{k-1}U_{k}^*,\;\; U_k^*A_k=S_kU_1^*.
$$
So
\begin{equation}\label{det:A:S}
\det(A_1A_2\ldots A_{k-1})(A_1A_2\ldots A_{k-1})^*=\det(S_1S_2\ldots S_{k-1})^2.\end{equation}
Now combining \eqref{eqn:step1}, \eqref{eqn:step2}, \eqref{eqn:step3}, \eqref{dsf:DS_i}, \eqref{eqn:stepk+2} and \eqref{det:A:S}, we get that
\begin{eqnarray}\label{eqn:jacobian}
\bigwedge_{i=1}^k|DA_i|&=&|\Delta(T)|^2\prod_{i=1}^{k-1}J(S_i)L(S_i)|\det(S_i)|^{-2n_1}\bigwedge_{i=1}^k|dH(U_i)||DS_i|
\bigwedge_{i=2}^k|DB_i|\nonumber\\
&=&\prod_{1\le i<j\le n_1}|z_i-z_j|^2\prod_{i=1}^{k-1}|\det(S_i)|^{2(n_{i+1}-n_1)}\bigwedge_{i=1}^k|dH(U_i)||DS_i|
\bigwedge_{i=2}^k|DB_i|,
\end{eqnarray}
where $z_1,z_2,\ldots,z_{n_1}$ are the eigenvalues of $A_1A_2\cdots A_k$.

\begin{remark}\label{re:tran}
In fact, the above transformations say
\begin{eqnarray*}
A_1&=&U_1S_1U_2^*\\
A_2&=&U_2S_2U_3^*+V_2B_2\\
&\vdots &\\
A_{k-1}&=&U_{k-1}S_{k-1}U_{k}^*+V_{k-1}B_{k-1}\\
A_{k}&=&U_{k}S_{k}U_{1}^*+V_{k}B_{k}.
\end{eqnarray*}
Following \textbf{steps} $1,2,\ldots, k+2$, one can  recover $U_1,U_2,\ldots,U_k,S_1,S_2,\ldots,S_{k}$ and $B_2,B_3,\ldots,B_k$ from $A_1,A_2,\ldots,A_k$.
\end{remark}

\begin{remark}\label{re:square}
Observe that if $A_1,A_2,\ldots,A_k$ are square matrices, then (\ref{eqn:jacobian}) takes the following form
\begin{equation}\label{eqn:square}
\bigwedge_{i=1}^k|DA_i|=\prod_{1\le i<j\le n_1}|z_i-z_j|^2\bigwedge_{i=1}^k|dH(U_i)||DS_i|
\end{equation}
where $|dH(U_i)|$ and $|DS_i|$ are as in \eqref{eqn:DS}.
\end{remark}

\noindent{\bf Discussion on QR-decomposition:}
QR-decomposition can be thought of as polar decomposition
for matrices. Any matrix $M\in \mathcal{M}_n$ can be written as
$$
M=QR
$$
where $Q$ is unitary matrix and $R$ is upper triangular matrix with non negative diagonal entries. Then
$$
M_j=\sum_{i=1}^jQ_iR_{i,j}
$$
where $M_j$ and $Q_j$ are $ j$-th columns of $M$ and $Q$ respectively. We would like to write Lebesgue measure on $M$ in terms of Haar measure on $Q$ and Lebesgue measure on $R$.
Since $M_{1}=R_{1,1}Q_1$,  so Lebesgue measure on $M_1$ is  given by
$$|DM_1|=R_{1,1}^{2n-1}dR_{1,1} d\sigma_{\mathbb{T}^n}(Q_1)$$
 where $d\sigma_{\mathbb{T}^n}$ denotes volume measure on unit sphere ($\mathbb{T}^n$) in $\mathbb{C}^n$. Once $Q_1$ is fixed, any new column $M_2$ can be written as
$$
M_2=R_{1,2}Q_1+R_{2,2}Q_2
$$
where $Q_2$ is unit vector orthogonal to $Q_1$ and $R_{2,2}\ge 0$. By unitary invariance of Lebesgue measure,  Lebesgue measure on $M_2$ can be written as
$$
|DM_2|=R_{2,2}^{2n-3}dR_{2,2}|dR_{1,2}|^2d\sigma_{\mathbb{T}^n\cap Q_1^{\perp}}(Q_2)
$$
where $Q_1^{\perp}$ is the space which is perpendicular to $Q_1$ and $d\sigma_{\mathbb{T}^n\cap Q_1^{\perp}}$ denotes volume measure on manifold ${\mathbb{T}^n\cap Q_1^{\perp}}$.
Continuing this way, Lebesgue measure on $M_i$,
$$
|DM_i|=R_{i,i}^{2(n-i+1)-1}dR_{i,i}|dR_{1,i}|^2|dR_{2,i}|^2\cdots|dR_{i-1,i}|^2d\sigma_{\mathbb{T}^n\cap \{Q_1,Q_2,\ldots,Q_{i-1}\}^{\perp}}(Q_i).
$$
Therefore
\begin{eqnarray*}
|DM|&=&\prod_{i=1}^n|DM_i|
\\&=&\l[\prod_{i=1}^nR_{i,i}^{2(n-i+1)-1}dR_{i,i}\r]\l[\prod_{i<j}|dR_{i,j}|^2\r]
\l[\prod_{i=1}^nd\sigma_{\mathbb{T}^n\cap \{Q_1,Q_2,\ldots,Q_{i-1}\}^{\perp}}(Q_i)\r].
\end{eqnarray*}
We can see that measure on $Q$ given by
$$
\l[\prod_{i=1}^nd\sigma_{\mathbb{T}^n\cap \{Q_1,Q_2,\ldots,Q_{i-1}\}^{\perp}}(Q_i)\r]
$$
 is Haar measure on unitary group $\mathcal{U}(n)$  $\l(dH_{\mathcal{U}(n)}(Q)\r)$.
 So, finally we have
 \begin{equation}\label{eqn:qrdecomposition}
|DM|=\l[\prod_{i=1}^nR_{i,i}^{2(n-i+1)-1}dR_{i,i}\r]\l[\prod_{i<j}|dR_{i,j}|^2\r]
|dH_{\mathcal{U}(n)}(Q)|.
\end{equation}

\noindent{\bf Discussion on QR-decomposition for $\mathcal{N}_{m,n}$:}
First recall from \eqref{n_m,n} that $\mathcal{N}_{m,n}=\{Y\in \mathcal{M}_n:Y_{i,j}=0,1\le j<i\le m \}$.    Any matrix $M\in \mathcal{N}_{m,n}$ can be written as
$$
M=QR
$$
where $Q$ is unitary matrix in $\mathcal{V}$ and  $R$ is upper triangular matrix with non negative diagonal entries. Then
$$
M_j=\sum_{i=1}^jQ_iR_{i,j}
$$
where $M_j$ and $Q_j$ are $j$-th columns of $M$ and $Q$ respectively. We would like to write Lebesgue measure on $M$ in terms of Haar measure on $Q$ and Lebesgue measure on $R$.

Note $M_{1}=R_{1,1}Q_1$ where $Q_1$ is unit vector orthogonal to $e_2,e_3,\ldots,e_m$.  So Lebesgue measure on $M_1$, $|DM_1|=R_{1,1}^{2(n-m+1)-1}dR_{1,1}
d\sigma_{\mathbb{T}^n\cap\{e_2,\ldots,e_m\}^\perp}(Q_1)$
where $e_1,e_2,\ldots,e_n$ are standard basis vectors in $\mathbb{C}^n$ and  $d\sigma_{\mathbb{T}^n\cap\{e_2,\ldots,e_m\}^\perp}$ denotes volume measure on manifold $\mathbb{T}^n\cap\{e_2,\ldots,e_m\}^\perp$ in $\mathbb{C}^n$. Once $Q_1$ is fixed, second column $M_2$ can be written as
$$
M_2=R_{1,2}Q_1+R_{2,2}Q_2
$$
where $Q_2$ is unit vector orthogonal to $Q_1,e_3,\ldots,e_m$, and $R_{2,2}\ge 0$. By unitary invariance of Lebesgue measure,  Lebesgue measure on $M_2$ can be written as
$$
|DM_2|=R_{2,2}^{2(n-m+1)}dR_{2,2}|dR_{1,2}|^2d\sigma_{\mathbb{T}^n\cap\{Q_1,e_3,\ldots,e_m\}^\perp}(Q_2).
$$
Continuing this way, Lebesgue measure on $M_i\;(i< m)$,
$$
|DM_i|=R_{i,i}^{2(n-m+1)-1}dR_{i,i}|dR_{1,i}|^2|dR_{2,i}|^2\cdots|dR_{i-1,i}|^2d\sigma_{\mathbb{T}^n\cap \{Q_1,\ldots,Q_{i-1},e_{i+1},\ldots,e_m\}^{\perp}}(Q_i).
$$
Lebesgue measure on $M_i\;(i\ge m)$,
$$
|DM_i|=R_{i,i}^{2(n-i+1)-1}dR_{i,i}|dR_{1,i}|^2|dR_{2,i}|^2\cdots|dR_{i-1,i}|^2d\sigma_{\mathbb{T}^n\cap \{Q_1,\ldots,Q_{i-1}\}^{\perp}}(Q_i).
$$
Therefore
\begin{eqnarray*}
|DM|&=&\prod_{i=1}^n|DM_i|
\\&=&\l[\prod_{i=1}^{m-1}R_{i,i}^{2(n-m+1)-1}dR_{i,i}\r]\l[\prod_{i=m}^nR_{i,i}^{2(n-i+1)-1}dR_{i,i}\r]
\l[\prod_{i<j}|dR_{i,j}|^2\r]
\\&&\l[\prod_{i=1}^{m-1}d\sigma_{\mathbb{T}^n\cap \{Q_1,\ldots,Q_{i-1},e_{i+1},\ldots,e_m\}^{\perp}}(Q_i)\r]\l[\prod_{i={m}}^nd\sigma_{\mathbb{T}^n\cap \{Q_1,Q_2,\ldots,Q_{i-1}\}^{\perp}}(Q_i)\r].
\end{eqnarray*}
We can see that measure given by
$$
\l[\prod_{i=1}^{m-1}d\sigma_{\mathbb{T}^n\cap \{Q_1,\ldots,Q_{i-1},e_{i+1},\ldots,e_m\}^{\perp}}(Q_i)\r]\l[\prod_{i={m}}^nd\sigma_{\mathbb{T}^n\cap \{Q_1,Q_2,\ldots,Q_{i-1}\}^{\perp}}(Q_i)\r]
$$
on $Q$ is Haar measure on   $\mathcal{V}$  $\l(dH_{\mathcal{V}}(Q)\r)$.
 So, finally we have
 \begin{eqnarray}\label{qr1}
|DM|
&=&\l[\prod_{i=1}^{m-1}R_{i,i}^{2(n-m+1)-1}dR_{i,i}\r]\l[\prod_{i={m}}^nR_{i,i}^{2(n-i+1)-1}dR_{i,i}\r]\nonumber
\\&&\l[\prod_{i<j}|dR_{i,j}|^2\r]
|dH_{\mathcal{V}}(Q)|.
\end{eqnarray}

\noindent{\bf  Discussion on manifold:}
In this part we state a useful formula (Co-area formula) on manifold.
Before stating the Co-area formula we need to introduce some notation.
Fix a smooth map $f:M\to N$ from an manifold of dimension $n$ to a manifold of dimension $k$. We denote  derivative of $f$ at $p\in M$  by
$$D_p(f):\mathbb{T}_p(M)\to \mathbb{T}_{f(p)}N.$$
We denote
$$M_{reg}:=  \mbox{ Set of regular points of } f,$$
$$J(D_p(f)):= \mbox{Generalized determinant of }
D_p(f),$$
$$\rho_M:= \mbox{
volume measure  on }M.$$
\begin{theorem}[{\bf Co-area formula}]\label{coarea}
With notation and setting as above, let $\phi $ be any non-negative Borel-measurable function on $M$. Then
\begin{enumerate}
\item The function $p\mapsto J(D_p(f)) $ on M is Borel-measurable.

\item The function $q\mapsto \phi(p)d\rho_{reg}\cap f^{-1(q)}(p)$ on N is Borel-measurable.

\item The integral formula:
\begin{equation}\label{eqn:coarea}
\int_M \phi(p)J(D_p(f))d\rho_{M}(p)=\int_N\l(\int\phi(p) d\rho_{M_{reg}\cap f^{-1}(q)}(p)\r)d\rho_N(q)
\end{equation}
holds.
\end{enumerate}
\end{theorem}
We shall use this formula in the proof of Theorem \ref{thm:truncatedunitary}. For the proof of Co-area formula see \cite{andersonbook}  (pp. 442).\\

\noindent{\bf Acknowledgement:}
We would like to thank  Manjunath Krishnapur for many valuable
suggestions and discussions.

\end{document}